%% file: BB4_paper.tex
\documentclass[12pt]{article}

\usepackage{mathptmx,amsmath,amssymb,amsfonts}
\usepackage{amsthm} 
\usepackage{graphicx,color,xcolor,pictex,epstopdf,url,algorithm,algorithmic,caption}
\usepackage{endnotes}

\usepackage{etoolbox} 
\makeatletter
\patchcmd{\@makechapterhead}{\large}{\normalsize}{}{}
\patchcmd{\@makechapterhead}{\large}{\normalsize}{}{}
\patchcmd{\@makeschapterhead}{\normalsize}{\normalsize}{}{}
\makeatother
\usepackage[textwidth=6.0in,textheight=9in,left=0.75in,right=0.75in,top=0.75in,bottom=0.75in]{geometry}
\usepackage[section]{placeins}
\makeatletter
\g@addto@macro\normalsize{\setlength\abovedisplayskip{4pt}}
\g@addto@macro\normalsize{\setlength\belowdisplayskip{4pt}}
\makeatother
\newtheorem{theorem}{Theorem}[section]
\newtheorem{lemma}{Lemma}[section]

\newtheorem{remark}{Remark}[section]
\usepackage{parskip}
\let\oldref\ref
\renewcommand{\ref}[1]{(\oldref{#1})}  
\renewcommand{\eqref}[1]{(\oldref{#1})} 
\font\titlefonrm=ptmb scaled \magstep3

\newbox\boxaddrone \newbox\boxaddrtwo

\def\half{{\frac{1}{2}}}
\def\Ltwo{2}
\def\Linf{\infty}

\begin{document}

\title{\titlefonrm Recovery of multiple coefficients in a reaction-diffusion equation}
\author{Barbara Kaltenbacher\footnote{
Department of Mathematics,
Alpen-Adria-Universit\"at Klagenfurt.
barbara.kaltenbacher@aau.at.}
\and
William Rundell\footnote{
Department of Mathematics,
Texas A\&M University,
Texas 77843. 
rundell@math.tamu.edu}
}
\maketitle

\begin{abstract}

This paper considers the inverse problem of recovering both the
unknown, spatially-dependent conductivity $a(x)$ and the potential $q(x)$
in a parabolic equation from overposed data consisting of the value of
solution profiles taken at a later time $T$.
We show both uniqueness results and the convergence of an iteration
scheme designed to recover these coefficients.
We also allow a more general setting, in particular when the usual
time derivative is replaced by one of fractional order and when
the potential term is coupled with a known nonlinearity $f$ of the form
$q(x)f(u)$.
\\[1ex]
\end{abstract}


\section{Introduction}\label{sect:introduction}

We consider the inverse problem of recovering coefficients from the uniformly
elliptic operator $-\mathbb{L}$ within a diffusion model.
In this case we will include both parabolic as well as anomalous diffusion
processes and the situation we describe will be general enough to include known
nonlinear reaction terms.
Reaction-diffusion equations such as these occur throughout the sciences
and we give some specific examples in the next section.
Let $\,\mathbb{L}u = -\nabla\cdot(a(x)\nabla u) + q(x)u$ be defined on a domain
$\Omega\subset\mathbb{R}^n$ with smooth boundary $\partial\Omega$
and  where the two coefficients
$a$ and $q$ are the quantities to be determined in the inverse problem.
In this setting our basic model equation is thus
\begin{equation}\label{eqn:basic_pde_parabolic}
u_t(x,t) + \mathbb{L}u(x,t) = r(x,t,u)
\end{equation}
where $r(x,t,u)$ is a known forcing function.
Extensions of the above are of course possible and we mention
the case of a reaction-diffusion model in which \ref{eqn:basic_pde_parabolic}
becomes
\begin{equation}\label{eqn:nonlin_pde}
u_t(x,t) -\nabla\cdot(a(x)\nabla u(x,t)) = q(x)f(u) + r(x,t,u)
\end{equation}
where the form of the nonlinear driving term $f(u)$ is assumed known.
Boundary conditions for \ref{eqn:basic_pde_parabolic} will be of the impedance form
\begin{equation}\label{eqn:bdry_cond}
a\frac{\partial u}{\partial\nu} + \gamma u = h,\qquad
x\in\partial\Omega,\quad t>0,
\end{equation}
and we impose the initial condition
\begin{equation}\label{eqn:init_cond}
u(x,0) = u_0(x),\qquad x\in\Omega.
\end{equation}
Examples of these models are in ecology where $u$ represents the population
density at a fixed point $x$ and time $t$ and $f(u)$ is frequently taken
to be quadratic in $u$ as in the Fisher model;
or in chemical reactions where $f$ is cubic in
the case of the combustion theory of Zeldovich and Frank-Kamenetskii,
\cite{Grindrod:1996, Murray:2002}.
Now of course the recovery of the coefficients $a$ and $q$ requires
over-posed data and we shall assume this is a spatial measurement at a fixed
time $T$ for two different sets of boundary conditions or the value of
a single solution at two different later times $t=T_1,\;T_2$.
Under different assumptions on the continuous time random walk {\sc ctrw}
model one obtains alternative diffusion processes and we consider the
subdiffusion model based on fractional time derivatives.
Now the basic equations take the form
\begin{equation}\label{eqn:basic_pde}
D^\alpha_t u(x,t) + \mathbb{L}u(x,t) = r(x,t,u)
\end{equation}
where $D_t^\alpha$ denotes the 
Djrbashian-Caputo derivative of order $\alpha$.
Setting all of the above in context, we want to also understand how
the different diffusive processes effect the ability to recover the
coefficients in the inverse problem.

\section{Background}

Undetermined coefficient problems based on the equation \ref{eqn:basic_pde}
have a long history in the literature.
In general, the over-posed data has taken one of two types:
the data suggested above, namely the spatial  values of $u(x,T)$ for fixed
time $T$; or time trace data, typically measured at discrete
points $\{x_i\}$ on the lateral boundary of the cylinder $\Omega\times(0,t)$
for $t>0$.

In the latter case if the boundary conditions in \ref{eqn:bdry_cond}
are of impedance type with $\gamma<\infty$ then this is taken to be Dirichlet
values and if $\gamma=\infty$, that is Dirichlet conditions are imposed
then the over-posed data is flux values at $\{x_i\}$.
The latter situation has been the most common, in particular
in one spatial dimension, beginning with the work of Cannon and of Pierce,
\cite{Cannon:1975,Pierce:1979} and
continuing in the fractional diffusion case by,
for example, \cite{CNYY:2009,RundellYamamoto:2018}.
The techniques used have mostly revolved around the eigenfunction
expansion of the solution (in the homogeneous case)
\begin{equation}\label{eqn:eigen_rep}
u(x,t) = \sum_{n=1}^\infty \langle u_0,\phi_n\rangle
E_{\alpha,1}(-\lambda_n t^\alpha)\,\phi_n(x)
\end{equation}
where $\{\lambda_n,\,\phi_n(x)\}$ are the eigenvalue/eigenfunctions of 
$-\mathbb{L}$ on $\Omega$ and $E_{\alpha,\beta}$ is the Mittag-Leffler function.
When $\alpha=1$ this recovers the usual exponential function leading to
the familiar parabolic solution.
This representation is based on the 
Djrbashian-Caputo derivative from the initial point $a$
$\,{}^C_a D^\alpha_t f  = I_a^\alpha \frac{df}{ds}\,$
where
$\;I_a^\alpha f(x) = \frac{1}{\Gamma(\alpha)}\int_a^x \frac{f(s)}{(x-s)^{1-\alpha}}\,ds\;$  is the Abel fractional integral operator.
The subdiffusion case based on this derivative is well-documented in the 
literature and for background of particular relevance to inverse problems
we refer to the papers
\cite{KaltenbacherRundell:2019a, KaltenbacherRundell:2019b, JinRundell:2015}.

Assuming the initial value $u_0$ is given, then evaluating \ref{eqn:eigen_rep}
at $x_i\in \partial\Omega$ from the over-posed data values
 gives a Dirichlet series which
(under specific circumstances) can lead to recovering the spectrum
$\{\lambda_n\}$ and certain norming constants of the eigenfunction.
This offers little in higher space  dimensions, but in $\mathbb{R}^1$
this conversion to an inverse Sturm-Liouville problem can lead to a
uniqueness proof and, in theory, a reconstruction algorithm.
However there are serious difficulties.
The inversion of the Dirichlet series to obtain $\{\lambda_n\}$
given the asymptotic form $\lambda_n \approx Cn^2$,
is an extremely ill-conditioned problem.
Thus the ability to effectively recover many eigenvalues is limited,
even with extremely small values of $t$ being measured.
In order to even accomplish this one must ensure that the initial
data is chosen so that $\langle u_0,\phi_n\rangle\not=0$ for any $n$.
Since we don't know the eigenfunctions this is difficult to guarantee
other than through an argument that this is expected to occur
with probability zero or using very special $u_0$ such as a delta distribution
as in \cite{CNYY:2009}.
In addition, from any collection of such spectral data one can
only determine a single coefficient of $-\mathbb{L}$.
Indeed, the Liouville transform shows that the entire operator $-\mathbb{L}$
can be mapped into one with only a composite potential term $Q(x)$ appearing
in such a manner that the original spectrum is preserved, \cite{CCPR1997}.

The specification of spatial information avoids many of these drawbacks
and again has been well studied.
 See \cite{Rundell:1987,Isakov:1991} for the parabolic case and the recovery
of the single coefficient $q(x)$ and also \cite{ZhangZhi:2017}
for the subdiffusion case.
In \cite{KaltenbacherRundell:2019b} this was generalized to include
a nonlinear term, namely $q(x)f(u)$ where $f(u)$ was known and the spatial
factor $q(x)$ had to be determined from initial and final data.
In the other direction \cite{KaltenbacherRundell:2019b} showed that a
reaction term $f(u)$ could be recovered from such data.
As far as the authors are aware the current work is the first attempt
at showing a uniqueness theorem and a reconstruction algorithm for the
case of two independent coefficients even in the parabolic case.
However in the case of the elliptic operators this has been accomplished
albeit in a limited setting.  See, for example, \cite{Isakov:2006}.

The outline of this paper is as follows.
In the next section we present the recovery algorithm for the
pair of coefficients and then proceed by giving conditions
that lead to both uniqueness and convergence.
The final sections shows some quite detailed numerical experiments
that show both the feasibility and the constraints of the method.

We will use a variety of spaces and norms during the analysis and make some
comments on their notation here.

By $C^{k,\beta}(\Omega)$ we mean the Schauder spaces of those functions
whose $k^{\mbox{{th}}}$ derivative is H\"older continuous
of order $\beta$ on the set $\Omega$.
The H\"older norm of a function being $\;\sup_{x\in\Omega}|f(x)| +
\sup_{x,y\in\Omega}\bigl((|f(x)|-|f(y)|)/(|x-y|)\bigr)$.

Sobolev spaces will be denoted by the usual  notation $H^{k,p}(\Omega)$, 
and in case of $p=2$ simply $H^{k}(\Omega)$.
However, we will have no use here for $p$ values other than $p=2$ or
$p=\infty$ and the latter almost always coupled with $k=0$ giving $L^\infty$.
Then $L^2 = H^{0,2}$ and $L^\infty = H^{0,\infty}$ norms will be denoted by
$\|\ldots\|_2$ and $\|\ldots\|_\infty$,
and the $L^2$ inner product by $\langle\cdot,\cdot\rangle$.
Additionally we will use the Hilbert spaces $\dot{H}^s(\Omega)$ defined by
$  \dot{H}^s(\Omega) = \left\{v\in L^2(\Omega): \sum_{j=1}^\infty \lambda_j^{s}|(v,\varphi_j)|^2<\infty\right\}$
with the norm
$  \|v\|_{\dot{H}^s(\Omega)}^2=\sum_{j=1}^\infty \lambda_j^{s}|(v,\varphi_j)|^2$, 
which is equivalent to the $H^s$ Sobolev norm for $s\in[0,2]$ for $a\in H^{1,\infty}(\Omega)$, 
$q\in L^2(\Omega)$, and $\mathbb{L}$ elliptic.

\def\ul#1{\underline{#1}}
\def\lamtil{\tilde{\lambda}}
\def\phitil{\tilde{\varphi}}
\section{Algorithms and their analysis}\label{sect:numer_algor}

In this section we develop the main algorithmic scheme to show
uniqueness for recovery of both $a(x)$ and $q(x)$ from the overposed
data under appropriate conditions.  This scheme is constructive being
iterative in nature with a convergence analysis possible.

\subsection{An iteration scheme to recover $a$ and $q$}\label{subsec:iter}

To recapitulate, we seek to determine both $a(x)$, $q(x)$ in 
\begin{equation}\label{eqn:uv}
\begin{aligned}
&D^\alpha_t u-\nabla\cdot(a\nabla u)+q\,f(u) = r_u \quad t\in(0,T)\,,
\qquad u(0) = u_0\\
&D^\alpha_t v-\nabla\cdot(a\nabla v)+q\,f(v) = r_v \quad t\in(0,T)\,,
 \qquad v(0)=v_0
\end{aligned}
\end{equation}
with prescribed impedance boundary and given initial conditions
\begin{equation}\label{eqn:init_bc}
a\,\partial_\nu u+\gamma u = s(x,t),\quad x\in\partial\Omega
\qquad u(x,0) = u_0,\quad v(x,0) = v_0
\end{equation}
and known forcing functions
$r_u = r_u(x,t,u)$, $r_v = r_v(x,t,v)$ and reaction term $f(u)$
from observations $g_u(x):=u(x,T)$, $g_v:=v(x,T)$.

For a given $a$, $q$ we can evaluate \eqref{eqn:uv} on the surface $t=T$
to obtain
\begin{equation}\label{eqn:uv_sys}
\begin{aligned}
-\nabla\cdot(a\nabla g_u) + q\, f(g_u) &= r_u - D_t^\alpha u(x,T;a,q) \\
-\nabla\cdot(a\nabla g_v) + q\, f(g_v) &= r_v - D_t^\alpha v(x,T;a,q) \\
\end{aligned}
\end{equation}
or 
\begin{equation}\label{eqn:uv_sys2}
\mathbb{M} \left[{a}\atop{q}\right] = \mathbb{F}(a,q)
\end{equation}
where $\mathbb{M}$ is a linear operator depending only the data $g_u$ and $g_v$
and its derivatives and $\mathbb{F}$ is a nonlinear operator on $(a,q)$.

The strategy will be to provide conditions under which $\mathbb{M}$ is invertible
and the combined  nonlinear operator
$\mathbb{T}(a,q) := \mathbb{M}^{-1} \mathbb{F}(a,q)$ is contractive.

To establish an iterative reconstruction scheme, we therefore
let $\mathbb{T}(a,q)=(a^+,q^+)$, where 
\begin{equation}\label{eqn:aplusqplus}
\begin{aligned}
-\nabla\cdot(a^+\nabla g_u)+q^+f(g_u) = r_u(T)-D^\alpha_t u(\cdot,T;a,q)\\
-\nabla\cdot(a^+\nabla g_v)+q^+f(g_v) = r_v(T)-D^\alpha_t v(\cdot,T;a,q)\\
\end{aligned}
\end{equation}
and $u(x,t):= u(x,t;a,q)$, $v(x,t):= v(x,t;a,q)$ solve equation~\eqref{eqn:uv}.

In the case $f(u)=u$ there is an obvious approach to the above; multiply
the first equation in
\eqref{eqn:uv_sys} by $g_v$, the second by $g_u$ and subtract thereby
eliminating $q$ from the left hand side.

This gives
\begin{equation}\label{eqn:a_W}
\nabla\cdot(a\,W) = a\nabla\cdot W + \nabla a\cdot W = \phi
\end{equation}
where
\begin{equation}\label{eqn:W_phi}
\begin{aligned}
W(x) &= g_v(x)\,\nabla g_u(x) - g_u(x)\,\nabla g_v(x)\\
\phi(x) &= g_v(x)\,D^\alpha_t u(x,T)- g_u(x)\,D^\alpha_t v(x,T) + g_v\,r_u(x,T,u(x,T)) - g_u\,r_v(x,T,v(x,T))\,.\\
\end{aligned}
\end{equation}
The value of $W$ on $\partial\Omega$ is known from the boundary conditions
imposed on the system so that \eqref{eqn:a_W} gives an update for $a(x)$ in
terms of $W$ and $\phi$.
This shows that the scheme \eqref{eqn:aplusqplus} can be inverted for $a$
and then by substitution, also for $q$ provided that $W$ does not vanish
in any subset of $\Omega$ with nonzero measure.

If now, for example, we are in one space dimension and
$u$ and $v$ share the same boundary conditions at the left endpoint, then
\begin{equation}\label{eqn:a_uncoupled}
a(x)W(x) = \int_0^x \phi(s)\,ds 
\end{equation}
where we used the fact that $W(0)=0$.
The above will be the basis of one implementation of our reconstruction process
in Section~\oldref{sect:recons}: namely the {\it eliminate $q$ version.}

This also works the other way around so
an alternative is to multiply the first equation by $\nabla g_v$ and the second
by $\nabla g_u$ and subtract giving
\begin{equation}\label{eqn:q_uncoupled}
\begin{aligned}
q(x) W(x) &= a(x)\nabla W(x) + \psi(x) \\
\psi(x) &= (r_u(x,T,u(x,T))- D^\alpha_t u(x,T))\nabla g_v(x) - (r_v(x,T,v(x,T))-D^\alpha_t v(x,T))\nabla g_u(x)\\
\end{aligned}
\end{equation}
where $\psi(x)$ has already been computed from the previous iteration.

There is a seeming symmetry between this uncoupling of $a$ and $q$ but this
first impression could be misleading.
In \eqref{eqn:a_uncoupled} we obtain an updated $a$ directly from previous
iteration values of both $a$ and $q$ and this involves only the function
$W$.
The inversion of $a$ will go smoothly if $W$ does not vanish in $\Omega$
and even zeros of measure zero can be handled as we will see in
Section~\oldref{sect:recons}.
In \eqref{eqn:q_uncoupled} we obtain an updated $q$ that also depends not only on
$W$ but also $\nabla W$.
As we shall see, this makes the uncoupling of $q$ less stable than the
other way around.

However, the important point is that the above shows that the linear operator
$\mathbb{M}$ can be inverted by eliminating either of $q$ or $a$.
Given this,
we could also use \eqref{eqn:uv_sys2} directly by inverting
the linear operator $\mathbb{M}$ and solving simultaneously for $a$ and $q$
after representing these functions in a basis set.
An implementation of this approach will also be shown in
Section~\oldref{sect:recons}
and in general turns out to be the most effective approach, more clearly
avoiding some of the difficulties noted above
by using a least squares setting.

\subsection{Contractivity in one space dimension: eigenfunction expansion}\label{isl_contractivity}

It is well known that the full Sturm-Liouville form
\begin{equation}\label{eqn:full_SL}
-(a(x)u_x)_x + q(x)u = \lambda r(x) u 
\qquad 0<x<1
\end{equation}
can be placed into canonical Schr\"odinger  form
\begin{equation}\label{eqn:canonical_SL}
-v_{yy} + Q(y)v = \mu v 
\qquad 0 < y<L
\end{equation}
where the form of the boundary conditions is preserved and the 
equivalent $Q$ is given by the classical Liouville transformation,
achieved by setting 
$Q(x) = \frac{f''}{f}(x) + L^2 q(x)$ where
$f(x) = [a]^{1/4}$, $\,L =\int^1_0 [a(s)]^{-1/2}\,ds$
and $\mu = L^2\lambda$,
see \cite{CCPR1997}.
However, we do not need such regularity assumptions and
the following version due to Everitt can be found in \cite{Everitt:2005}.

\begin{lemma}\label{lem:Liouville_transform}
Let $a$, $r$ be such that $a$, $a'$, $r$, $r'$ be absolutely continuous on
$(0,1)$ with $a(x)$ and $r(x)$ strictly positive.
Then the mapping
$$
\ell(x) = \int_0^x [ r(s)/a(s)]^{1/2}\,ds
$$
taking $(0,1)$ into $(0,L)$
has an inverse mapping $\ell^{-1}$ where $L = \int_0^1 [ r(s)/a(s)]^{1/2}\,ds$.
Then the Sturm-Liouville equation \eqref{eqn:full_SL} can be transformed
into \eqref{eqn:canonical_SL} by the Liouville transformation
$\,y \to \ell(x)$, $\;v(y) \to v(\ell(x)) = [a(x)r(x)]^{1/4}u(x)$
where
$$
Q(y) = [r(x)]^{-1}q(x) - [a(x)/r^3(x)]^{1/4}
\bigl[a(x)\bigl([a(x)r(x)]^{-1/4}\bigr)'\bigr]'
$$
\end{lemma}

\begin{remark}\label{rem:r_not_1}
We will of course take $r(x)=1$ as it is not involved in our current problem
and this makes the full transformation simpler.
We note that the coefficients $a$ and $r$ play a complementary role and
the more general case can be used if indeed our inverse problem was to
determine the specific heat $r(x)$ and the potential $q(x)$
in the parabolic equation $r(x)u_t - u_{xx} + q(x)u$.
\end{remark}

\begin{remark}\label{rem:isl_weaker}
One can extend the spaces involved even further by accepting potentials
$Q(x)$ in the distribution space $H^{-1}(0,1)$ and still retain the
essential features required of the eigenvalues/eigenfunctions,
see \cite{HrynivMykytyuk:2003}.
This in turn allows weaker assumptions to be placed on the coefficient $a$:
for example, to lie in $H^1(0,1)$.
See, \cite{RundellSacks:1992b}.
\end{remark}

Now let $Q_1$ and $Q_2$ be in $L^2[0,1]$ 
and let
$\{\hat{\phi}_n(x),\lambda_n\}_1^\infty$ and $\{\hat{\psi}_n(x),\mu_n\}_1^\infty$ 
denote the corresponding eigenfunction/eigenvalue pairs:
\begin{equation}\label{eqn:eigenfunc}
-\hat{\phi}''_n + Q_1(x)\hat{\phi}_n = \lambda_n\hat{\phi}_n,
\qquad
-\hat{\psi}''_n + Q_2(x)\hat{\psi}_n = \mu_n\hat{\psi}_n,
\end{equation}
where we choose the normalizations $u'(0) = 1$ if the left boundary condition
is Dirichlet and $u(0)=1$ in the case of an impedance boundary condition
with a finite impedance value.

\begin{lemma}\label{lem:eigenvalue_lipschitz}
For some constant $C = C(M_Q)$
\begin{equation}\label{eqn:eigenvalue_lip}
|\lambda_n - \mu_n| \leq C\|Q_1 - Q_2\|_2
\end{equation}
\end{lemma}

\begin{proof}
We have the asymptotic expansion \cite{CCPR1997}
$$
\lambda_n = n^2\pi^2 + \int^1_0\!Q(t)\,dt - \int^1_0\!Q(t)\cos2n\pi t\,dt
+ t_n,\qquad \{t_n\}\in \ell^2
$$
in the case of Dirichlet conditions where each subsequent $k^{th}$
term in the expansion is the inner product of $Q$ with $\cos 2k\pi x$
and the decay rate of each term is as indicated.
The estimate \eqref{eqn:eigenvalue_lip} now easily follows

In fact this above asymptotic rate holds for $Q$ only
in $L^2(0,1)$ but the term $t_n$ is replaced by $\eta_n\in\ell^2$.
In the case of non-Dirichlet conditions the cosine terms are replaced by sines.
\end{proof}

Suppose that $P(x)$ and $Q(x)$ are two potentials in $L^2(0,1)$
with the same spectrum $\{\lambda_n\}$: that is,
$-u'' + Q u = \lambda u$ and $-v'' + P v = \lambda v$ where $u$ and $v$ share
the same conditions at $x=0$.
We will take these to be Dirichlet and impose the normalization that
$u'(0)=v'(0) = 1$.
Then it follows from the well-known Gel'fand-Levitan formulation that
\begin{equation}\label{eqn:GL}
v(x) = u(x) + \int_0^x K(x,t)u(t)\,dt
\end{equation}
where
\begin{equation}\label{eqn:K_hyperbolic}
\begin{aligned}
K_{tt}-K_{xx}+\bigl(Q(x)-P(x)\bigr)K &= 0, \quad 0\leq t\leq x \leq 1 \\
K(x,\pm x) &= \pm\half \int_0^x \bigl(Q(s) - P(s)\bigr)\,ds 
\end{aligned}
\end{equation}
In the case of impedance boundary conditions at $x=0$,
$u'(0) -h u(0) = 0$  we would normalize by
$u(0s)=1$ and (3) above would be replaced by
$$
K(x,\pm x) = h + \half \int_0^x \bigl(Q(s) - P(s)\bigr)\,ds 
$$
The key point here is that $K(x,t)$ does not depend on $\lambda$,
 only on $P$ and $Q$.
It satisfies a hyperbolic partial differential equation
with characteristics given by the lines $x=\pm t$.
The boundary conditions in the above form a {\it Goursat problem} for the
hyperbolic equation.

Suppose now that $P(x)=0$ and we have \eqref{eqn:eigenfunc}.
For definiteness we assume Dirichlet conditions at $x=0$.

Then
$$
\hat{\phi}_n(x) = {{\sin(\sqrt{\lambda_n}x)}\over {\sqrt{\lambda_n}}} +
\int_0^x K(x,t){{\sin(\sqrt{\lambda_n}t)}\over {\sqrt{\lambda_n}}}\,dt.
$$
and the solution to \eqref{eqn:K_hyperbolic} can be written as
\begin{equation}\label{eqn:K_represent}
K(x,t) + \int_R Q(r) K(r,s)\,dr\,ds
= \int_{\frac{x-t}{2}}^{\frac{x+t}{2}} Q(s)\,ds
\end{equation}
where $R$ is the rectangular region with corners $(x,t)$,
$(\frac{x-t}{2},\frac{x-t}{2})$,
$(\frac{x+t}{2},\frac{x+t}{2})$,
and $(0,0)$,
From the Volterra equation \eqref{eqn:K_represent}
it follows that $K$ depends on $Q$ through
a Lipschitz bound $\|K(\cdot;Q_1) - K(\cdot;Q_2)\|_\infty \leq C \|Q_1-Q_2\|_2$.
Indeed, if $Q>0$ then we have immediately from the above that
$\|K(\cdot;Q_1) - K(\cdot;Q_2)\|_\infty \leq \|Q_1-Q_2\|_2$.

Putting the above together 
and noting the correspondence 
$\phi_n=\sqrt{2\lambda_n}\hat{\phi}_n$,
we have the lemma
\begin{lemma}\label{lem:eigenfunction_lipschitz}
Given the above, the eigenfunctions corresponding to $Q_1$ and $Q_2$
must satisfy
\begin{equation}\label{eqn:efunc_lipschitz}
\begin{aligned}
\|\phi_n - \psi_n\|_\infty &\leq \sup_{0\leq x\leq 1}|\sin(\sqrt{\lambda_n}x) - \sin(\sqrt{\mu_n}x)|
+ \sup_{0\leq t\leq x}| K(x,t;Q_1) - K(x,t;Q_2)|\\
&\leq C\|Q_1 - Q_2\|_2
\end{aligned}
\end{equation}
\end{lemma}

The point of all of this is that the given the regularity assumptions on
functions
$a_1(x)$, $a_2(x)$, $q_1(x)$, $q_2(x)$, the transformed functions
$Q_1$ and $Q_2$ must satisfy
\begin{equation}\label{eqn:Q_Lipschitz}
\|Q_1 - Q_2\|_2
\leq C\bigl[ \|a_1 - a_2\|_{C^{1,\beta}} + \|q_1 - q_2\|_2\bigr]
\end{equation}
for any $\beta>0$, and we can therefore apply the above lemmas to see
that without loss of generality
we may assume that both the original  eigenfunctions and eigenvalues
must depend on a Lipschitz manner on $a_1-a_2$ and $q_1-q_2$.
\bigskip
Now returning to the solution representation \eqref{eqn:eigen_rep} of $u(x,t)$
$$
u(x,t) = \sum_{n=1}^\infty \langle u_0,\phi_n\rangle
E_{\alpha,1}(-\lambda_n t^\alpha)\,\phi_n(x)
$$
we see that for $t>0$ and the parabolic case 
\begin{equation}\label{eqn:u_t_rep_1}
u_t(x,t) = \sum_{n=1}^\infty b_n e^{-\lambda_n t}\,\phi_n(x),
\qquad b_n = -\lambda_n \langle u_0,\phi_n\rangle \,.
\end{equation}
This converges uniformly for $t\geq t_0>0$ and shows $u_t(x,t)$ for fixed $t$
is Lipschitz continuous in both the eigenfunctions $\{\phi_n\}$
and eigenvalues $\{\lambda_n\}$ and hence in light of the above
also in the functions $q$ and $a$ in the stated norms.
In addition, the exponential decay of the term $e^{-\lambda_n T}$ shows that
for $T$ sufficiently large, this Lipschitz constant will be less than unity.

In the case of the fractional time operator we need a modification.
There is no longer exponential decay of the solution as the function
$E_{\alpha,1}(-\lambda T^\alpha)$ has only linear decay in time and is a
well-known difference between the classical and fractional cases
especially in regards to inverse problems,
\cite{JinRundell:2015,KaltenbacherRundell:2019a}.

Now taking the time derivative of the solution means that the sequence
$\{b_n\}$ defined in \eqref{eqn:u_t_rep_1} must be assumed to be
at least in $L^2(\Omega)$ and in consequence we must add the additional
regularity assumption that the initial data $u_0\in H^2(\Omega)$.
This follows directly from the fact that multiplying by $\lambda$ in the
Fourier coefficients is equivalent to taking two derivatives in space.

We now formally show these statements below.

The difference $\mathbb{F}(a_1,q_1)-\mathbb{F}(a_2,q_2)$ can be decomposed
as follows
\[
\begin{aligned}
&u_t(x,T;Q_1) - u_t(x,T;Q_2)
= - \sum_{n=1}^\infty 
\Bigl(\lambda_n e^{-\lambda_n T} \langle u_0,\phi_n\rangle \phi_n(x) -
\mu_n e^{-\mu_n T} \langle u_0,\psi_n\rangle \psi_n(x)\Bigr)\\
&= - \sum_{n=1}^\infty \left\{
\left(\lambda_n e^{-\lambda_n T}-\mu_n e^{-\mu_n T}\right) \langle u_0,\phi_n\rangle
+\mu_n e^{-\mu_n T} \langle u_0,\phi_n- \psi_n\rangle\right\}
\phi_n(x)\\
&\quad-\sum_{n=1}^\infty \mu_n e^{-\mu_n T} \langle u_0,\psi_n\rangle \left(\phi_n(x) - \psi_n(x)\right)\,.
\end{aligned}
\]
This allows an estimate, for example, of the supremum norm of the difference by 
\[
\begin{aligned}
\|u_t&(x,T;Q_1)- u_t(x,T;Q_2)\|_{\Linf}\\
&\leq C_{\dot{H}^\sigma\to C}^\Omega 
\, \sup_{n\in\mathbb{N}} \lambda_n^{\sigma/2}\left|e^{-\lambda_n T}-\tfrac{\mu_n}{\lambda_n}e^{-\mu_n T}\right| \|u_0\|_{\dot{H}^2(\Omega)}\\
&\quad +C_{\dot{H}^\sigma\to C}^\Omega \Bigl(\sum_{n=1}^\infty \lambda_n^\sigma \left(e^{-\mu_n T}\right)^2
+\sum_{n=1}^\infty \lambda_n^{\sigma-2} \left(\mu_n e^{-\mu_n T}\right)^2\Bigr)^{\frac{1}{2}} \|u_0\|_{\dot{H}^2(\Omega)}\,
\ \sup_{n\in\mathbb{N}} \|\phi_n- \psi_n\|_{\Ltwo}\\
&\quad + C_{\dot{H}^\sigma\to C}^\Omega C_{\dot{H}^2\to L^\infty}^\Omega 
\Bigl(\sum_{n=1}^\infty \lambda_n^\sigma \left(e^{-\mu_n T}\right)^2\Bigr)^{\frac{1}{2}} \|u_0\|_{\dot{H}^2(\Omega)} \|Q_1-Q_2\|_{\Ltwo}\\
&\quad+\sum_{n=1}^\infty  \mu_n e^{-\mu_n T} |\langle u_0,\psi_n\rangle| 
\ \sup_{n\in\mathbb{N}} \|\phi_n- \psi_n\|_{\Linf}
\end{aligned}
\]
for $\sigma>1/2$, where we have used the fact that 
\[
\|\sum_{n=1}^\infty \mu_n e^{-\mu_n T} \langle u_0,\phi_n- \psi_n\rangle
\phi_n\|_{\dot{H}^\sigma(\Omega)}=
\Bigl(\sum_{n=1}^\infty \lambda_n^\sigma \Bigr(\mu_n e^{-\mu_n T}\Bigr)^2 \langle u_0,\phi_n- \psi_n\rangle^2 \Bigr)^{\frac{1}{2}}
\]
with
\[
\langle u_0,\phi_n\!-\! \psi_n\rangle^2 = \Bigl(\sum_{j=1}^\infty \lambda_j\langle u_0,\phi_j\rangle \frac{1}{\lambda_j}\langle \phi_j,\phi_n\!-\! \psi_n\rangle \Bigr)^2
\leq \|u_0\|_{\dot{H}^2(\Omega)}^2 \ \sum_{j=1}^\infty \frac{1}{\lambda_j^2}\langle \phi_j,\phi_n\!-\! \psi_n\rangle^2 \]
and, using $\langle \phi_j,\phi_n\rangle=0$ for $j>n$, 
\[
\begin{aligned}
\sum_{j=1}^{n-1} \frac{1}{\lambda_j^2} \langle \phi_j,\phi_n\!-\! \psi_n\rangle^2 
&=-\sum_{j=1}^{n-1} \frac{1}{\lambda_j^2}\langle \phi_j,\psi_n\rangle^2 
=\frac{-1}{\mu_n^2} \sum_{j=1}^{n-1} \langle (-\triangle +Q_1)^{-1}\phi_j, (-\triangle +Q_2)\psi_n\rangle^2 \\
&=\frac{1}{\mu_n^2} \sum_{j=1}^{n-1} \langle \phi_j, \phi_n-\psi_n + (-\triangle +Q_1)^{-1}[(Q_1-Q_2)\psi_n]\rangle^2 \\
&\leq \frac{1}{\mu_n^2} \Bigl(\|\phi_n-\psi_n\|_{\Ltwo}+C_{\dot{H}^2\to L^\infty}^\Omega \|Q_1-Q_2\|_{\Ltwo}\Bigr)^2\\[1ex]
\sum_{j=n}^\infty \frac{1}{\lambda_j^2}\langle \phi_j,\phi_n- \psi_n\rangle^2
&\leq \frac{1}{\lambda_n^2}\|\phi_n- \psi_n \|_{\Ltwo}^2
\end{aligned}
\]

Higher norms, which  are needed to recover $a$ in $C^{1,\beta}(\Omega)$ can be
estimated by using continuity of the embedding
$\dot{H}^\sigma(\Omega)\to C^{1,\beta}(\Omega)$.
In spite of the arising additional powers of $\lambda_n$,
the exponentially decaying factors will always dominate,
 even if the $C^k(\Omega)$ norm with arbitrary $k\in\mathbb{N}$ is taken. 
This is not the case for the Mittag-Leffler function, where due to the
estimate $\frac{1}{1+\Gamma(1-\alpha)x}\leq E_{\alpha,1}(-x)$ for all $x>0$,
the attainable smoothness is limited to $C^{1,\beta}$ with $\beta<1/2$.

\subsection{Contractivity from the pde directly}

We will reproduce the contractivity estimates from the previous section
but without direct reference to the eigenfunction expansion
\eqref{eqn:u_t_rep_1} since in higher space dimensions the relationship
between the coefficients and the eigenfunctions/eigenvalues is much less
clear. 
Instead we will use the differential equations directly for the differences
of both $a_1$, $a_2$ and $q_1$, $q_2$.
This leads to a non-homogeneous parabolic/subdiffusion system where the
right hand side depends on the coefficients themselves.
Caution is needed at several places as even in the parabolic case there are
difficulties.  For example, in a strong solution interpolation there are
regularity concerns.
For a right hand side function
$F\in C^{k,\beta}(\Omega)\times C^{r,\beta}(0,T)$ where $F$ contains the values
of $a$ and $q$, the corresponding solutions $u_t$ will lie in 
$C^{k,\beta}(\Omega)\times C^{r,\beta/2}(0,T)$ and this regularity drop
causes difficulties with the mapping properties of $\mathbb{T}$.
On the other hand, weak solutions in Sobolev spaces with, say,
$q\in L^2(\Omega)$ and $u\in H^1(\Omega)$ have the issue that 
$q\,u$ is undefined in $\mathbb{R}^d$ with $d>3$ and other embedding estimates
needed often make further restrictions unless $d=1$.
This is exactly the situation even in the case of an unknown potential
as shown in \cite{KaltenbacherRundell:2019b} and clearly is more complex
with a conductivity $a$ involved.

For two different coefficient pairs $(a,q)$, $(\tilde{a},\tilde{q})$ (with corresponding solutions $u$, $v$, $\tilde{u}:= u(x,t;\tilde{a},\tilde{q})$, $\tilde{v}:= v(x,t;\tilde{a},\tilde{q})$), the difference $(da^+,dq^+)=\mathbb{T}(a,q)-\mathbb{T}(\tilde{a},\tilde{q})$ satisfies
\begin{equation}\label{eqn:dadq}
\begin{aligned}
-\nabla(da^+\nabla g_u)+dq^+\,g_u = -D^\alpha_t u(T)+D^\alpha_t \tilde{u}(T):=-D^\alpha_t \hat{u}(T)\\
-\nabla(da^+\nabla g_v)+dq^+\,g_v = -D^\alpha_t v(T)+D^\alpha_t \tilde{v}(T):=-D^\alpha_t \hat{v}(T)
\end{aligned}
\end{equation}
where $\hat{u}$, $\hat{v}$ solve 
\begin{equation}\label{eqn:uvhat}
\begin{aligned}
&D^\alpha_t \hat{u}-\nabla(a\nabla \hat{u})+q\hat{u} = \nabla(da\nabla \tilde{u})-dq\,\tilde{u} \quad t\in(0,T)\,, \quad \hat{u}(0)=0\\
&D^\alpha_t \hat{v}-\nabla(a\nabla \hat{v})+q\hat{v} = \nabla(da\nabla \tilde{v})-dq\,\tilde{v} \quad t\in(0,T)\,, \quad \hat{v}(0)=0
\end{aligned}
\end{equation}
with $da=a-\tilde{a}$, $dq=q-\tilde{q}$.

The system \eqref{eqn:dadq} of PDEs for $da^+$ and $dq^+$ and its stable solution has been considered in different contexts, see, e.g.,  \cite{BalUhlmann2013}, see also \cite{EKN89} and \cite{HNS95}, for the single coefficient case.

Multiplying the first equation in \eqref{eqn:dadq} by $g_v$, the second one with $g_u$, and subtracting, we get, with
\begin{equation}\label{eqn:W}
W=g_v\nabla g_u-g_u\nabla g_v, 
\end{equation}
and using some cancellation leading to 
$\nabla(da^+ W)=\nabla da^+\cdot W + da^+(g_v\triangle g_u-g_u\triangle g_v)$ that
\begin{equation}\label{eqn:PDEda}
\nabla(da^+ W) = g_vD^\alpha_t \hat{u}(T)-g_uD^\alpha_t \hat{v}(T)\,.
\end{equation}

While considering the simultaneous identification of $a$ and $q$, we restrict ourselves to one space dimension $\Omega=(0,L)$, since this more easily allows to resolve \eqref{eqn:PDEda} and, in the computation of $dq^+$, to eliminate $\nabla da$; 
\footnote{space derivatives will therefore simply denoted by a prime, whereas, to avoid additional notation, we will stay with the nabla notation when it is a partial space derivative of a space and time dependent function}
Moreover, we will need this restriction on the space dimension for being able to use the embedding $H^1(\Omega)\to L^\infty(\Omega)$. Later on we will consider the identification of the potential $q$ alone, which works in higher space dimensions as well, see Section \oldref{sec:potentialonly}.
Additionally, to prove decay of $\tilde{u}$ as needed for establishing contractivity, we will focus on the homogeneous case $r_u=r_v=0$ and assume that $W$ is bounded away from zero.

\begin{lemma}
Let $\Omega=(0,L)\subseteq\mathbb{R}$, $r_u=r_v=0$, $\alpha\in (0,1)$ and let $g_u,g_v\in H^2(\Omega)$ satisfy $|\frac{1}{W(x)}|\leq C_0$ for $W$ defined in \eqref{eqn:W}.
Then any solution $(da^+,dq^+)$ of \eqref{eqn:dadq} with $da^+(0)=0$ satisfies the estimate 
\begin{equation}\label{eqn:estdadq}
\begin{aligned}
\|da^+\|_{\Linf}\leq&
C_0\left(\|g_v\|_{\Ltwo}\|D^\alpha_t \hat{u}(T)\|_{\Ltwo}+\|g_u\|_{\Ltwo}\|D^\alpha_t \hat{v}(T)\|_{\Ltwo}\right)\\
\|{da^+}'\|_{\Ltwo}\leq&C_0^2\|W'\|_{\Ltwo}\left(\|g_v\|_{\Ltwo}\|D^\alpha_t \hat{u}(T)\|_{\Ltwo}+\|g_u\|_{\Ltwo}\|D^\alpha_t \hat{v}(T)\|_{\Ltwo}\right)\\
&\quad +C_0\left(\|g_v\|_{\Linf}\|D^\alpha_t \hat{u}(T)\|_{\Ltwo}+\|g_u\|_{\Linf}\|D^\alpha_t \hat{v}(T)\|_{\Ltwo}\right)\\
\|dq^+\|_{\Ltwo}\leq&C_0\left(\|g_v'\|_{\Linf}\|D^\alpha_t \hat{u}(T)\|_{\Ltwo}+\|g_u'\|_{\Linf}\|D^\alpha_t \hat{v}(T)\|_{\Ltwo}\right)
+\|W'\|_{\Ltwo}\|da^+\|_{\Linf}\,.\\
\end{aligned}
\end{equation}
\end{lemma}

\begin{proof}
From \eqref{eqn:PDEda}
\[
\begin{aligned}
da^+(x)=&\frac{1}{W(x)}\left(da^+(0) W(0) +\int_0^x(g_v(s)D^\alpha_t \hat{u}(s,T)-g_u(s)D^\alpha_t \hat{v}(s,T))\, ds\right)\\
{da^+}'(x)=&-\frac{W'(x)}{W(x)^2}\left(da^+(0) W(0) +\int_0^x(g_v(s)D^\alpha_t \hat{u}(s,T)-g_u(s)D^\alpha_t \hat{v}(s,T))\, ds\right)\\
&+\frac{1}{W(x)}(g_v(x)D^\alpha_t \hat{u}(x,T)-g_u(x)D^\alpha_t \hat{v}(x,T))\,.
\end{aligned}
\]

Multiplying the first equation in \eqref{eqn:dadq} by $g_v'$,  the second one with $g_u'$, and subtracting, we can analogously eliminate ${da^+}'$ to get, with
$\tilde{W} = g_v g_u''- g_u g_v''=W'$
\[
dq^+(x)=\frac{1}{W(x)}\left(g_v'(x)D^\alpha_t \hat{u}(x,T)-g_u'(x)D^\alpha_t \hat{v}(x,T)-W'da^+\right)
\]
This yields the estimate \eqref{eqn:estdadq}.
\end{proof}
\begin{remark}
From \eqref{eqn:estdadq} it seems that the problem of recovering $q$ is more ill-posed since the estimates for $dq^+$ -- even in a weaker norm -- require higher derivatives of the data $g_u$, $g_v$. This is confirmed by the computational results.

Estimate \eqref{eqn:estdadq} could obviously as well be obtained by assuming $da^+(L)=0$ instead of $da^+(L)=0$. To achieve that $da^+$ vanishes at one of the boundary points, we assume that $a(0)=:a_0$ or $a(L)=:a_L$ is known and prescribe this value as an additional condition in \eqref{eqn:aplusqplus}.
\end{remark}

To estimate  $\|D^\alpha_t \hat{u}(T)\|_{\Ltwo}$ (and likewise $\|D^\alpha_t \hat{v}(T)\|_{\Ltwo}$) appearing in the right hand side of \eqref{eqn:estdadq}, we use eigensystems $(\lambda_j, \varphi_j)$ $(\lamtil_j,\phitil_j)$ of the operators defined by $\mathbb{L}w=-\nabla\cdot(a\nabla w)+q\,w$ and $\widetilde{\mathbb{L}}w=-\nabla\cdot(\tilde{a}\nabla w)+\tilde{q}w$, to obtain the representations
\begin{eqnarray}
\hat{u}(x,t)&=&\sum_{j=1}^\infty \int_0^t s^{\alpha-1}E_{\alpha,\alpha}(-\lambda_j s^\alpha)\langle\nabla\cdot(da\nabla\tilde{u}(t-s))-dq\,\tilde{u}(t-s),\varphi_j\rangle\, ds \varphi_j(x)
\label{eqn:repr_uhat}\\
\tilde{u}(x,t)&=&\sum_{j=1}^\infty E_{\alpha,1}(-\lambda_j t^\alpha)\langle u_0,\phitil_j\rangle \phitil_j(x)\,,
\label{eqn:repr_util}
\end{eqnarray}
where we have assumed $r_u=0$ in order to obtain a proper decay of $D^\alpha_t\tilde{u}(t)$.

We first of all provide an estimate of the right hand side of the equation \eqref{eqn:uvhat} for $\hat{u}$, i.e., of the inhomogeneity in \eqref{eqn:repr_uhat}.

For this purpose we will make use of the Poincar\'{e}-Friedrichs type inequality 
\begin{equation}\label{eqn:PF}
\|w\|_{\Ltwo}^2\leq C_{PF\Omega} \|\nabla w\|_{\Ltwo}^2 + C_{PF\partial\Omega} \int_{\partial\Omega}\gamma w^2\, ds\quad \mbox{ for all }w\in H^1(\Omega)
\end{equation}
and the assumptions
\begin{equation}\label{eqn:boundsaq}
\tilde{a}(x)\geq 2\ul{a}>0,\qquad 
\|\tilde{q}-\ul{q}\|_{\Ltwo}\leq \frac{\varrho}{(C_{H^1\to L^\infty}^\Omega)^2}
\end{equation} 
for some constant $\ul{q}$ (not necessarily positive) with
\begin{equation}\label{eqn:ulaulqvarrho}
\varrho\leq\ul{a}\mbox{ and }\Bigl( \varrho\leq\ul{q} \mbox{ or }
\varrho\geq\max\{\ul{q}+\frac{1}{C_{PF\partial\Omega}}, \frac{\ul{a}+C_{PF\Omega}}{1+C_{PF\Omega}}\}\Bigr)\,,
\end{equation}
as well as 
\begin{equation}\label{eqn:smallnessaprime}
\|\tilde{a}'\|_{\Ltwo}\leq \frac{\ul{a}}{C_{H^1\to L^\infty}^\Omega \|(-\triangle+\mbox{id})^{-1}\|_{L^2(\Omega)\to H^2(\Omega)}}\,.
\end{equation} 

\begin{lemma}\label{lem:util}
The function $\tilde{u}$ defined by \eqref{eqn:repr_util}, with \eqref{eqn:boundsaq}, \eqref{eqn:ulaulqvarrho}, \eqref{eqn:smallnessaprime} and $\widetilde{\mathbb{L}}u_0\in L^2(\Omega)$ satisfies the estimate 
\begin{equation}\label{eqn:estrhs1}
\begin{aligned}
&\|\nabla\cdot(da\nabla D^\alpha_t \tilde{u}(t))-dq\,D^\alpha_t \tilde{u}(t)\|_{\Ltwo}\\
&\leq C_2 \, e(\lamtil_1,t) \, \|\widetilde{\mathbb{L}}u_0\|_{\Ltwo} 
\Bigl(\|da'\|_{\Ltwo}+\|da\|_{\Linf}+\|dq\|_{\Ltwo}\Bigr)
\end{aligned}
\end{equation}
for some $C_2$ depending only on $\Omega$ and $\frac{1}{\ul{a}}$, where 
$\;{\displaystyle
e(\lamtil_1,t):=\sup_{\mu\geq\lamtil_1} \max\{1,\mu\} E_{\alpha,1}(-\mu t^\alpha)\,.
}$
\end{lemma}

\begin{proof}
We use the fact that $\tilde{u}$ solves the first of the two equations \eqref{eqn:uv} with $a=\tilde{a}$ and $r_u=0$, as well as homogeneous impedance boundary conditions and start by estimating
\begin{equation}\label{eqn:estrhs}
\begin{aligned}
&\|\nabla\cdot(da\nabla D^\alpha_t \tilde{u}(t))-dq\,D^\alpha_t \tilde{u}(t)\|_{\Ltwo}\\
&\leq
\|da'\|_{\Ltwo} \|\nabla D^\alpha_t \tilde{u}(t)\|_{\Linf} + \|da\|_{\Linf}\|\triangle D^\alpha_t \tilde{u}(t)\|_{\Ltwo}+ \|dq\|_{\Ltwo} \|D^\alpha_t \tilde{u}(t)\|_{\Linf}\\
&\leq
C_1(\|da'\|_{\Ltwo} + \|da\|_{\Linf}) 
\Bigl(\|\triangle D^\alpha_t \tilde{u}(t)\|_{\Ltwo}+\|D^\alpha_t \tilde{u}(t)\|_{\Ltwo}\Bigr)\\
&\quad+ \|dq\|_{\Ltwo} C_{H^1\to L^\infty}^\Omega\Bigl(\|\nabla D^\alpha_t \tilde{u}(t)\|_{\Ltwo}+\|D^\alpha_t \tilde{u}(t)\|_{\Ltwo}\Bigr)
\end{aligned}
\end{equation}
for $C_1=\max\{1,C_{H^1\to L^\infty}^\Omega \|(-\triangle+\mbox{id})^{-1}\|_{L^2(\Omega)\to H^2(\Omega)}\}$.
For the $\nabla D^\alpha_t \tilde{u}(t)$ term in \eqref{eqn:estrhs},
from the fact that $\tilde{z}:=D^\alpha_t \tilde{u}$ solves 
\begin{equation}\label{eqn:ztil}
D^\alpha_t \tilde{z}-\nabla(\tilde{a}\nabla \tilde{z})+\tilde{q}\tilde{z} = 0 \quad t\in(0,T)\,, \quad
\tilde{z}(0)=D^\alpha_t\hat{u}(0)=-\mathbb{L}u_0
\end{equation}
with impedance boundary conditions and using integration by parts, we get
\[
\begin{aligned}
\int_\Omega \tilde{a} |\nabla D^\alpha_t \tilde{u}(t)|^2\, dx 
=& -\int_\Omega \nabla\cdot (\tilde{a} \nabla D^\alpha_t \tilde{u}(t)) D^\alpha_t \tilde{u}(t)\, dx 
+\int_{\partial\Omega}\tilde{a}\partial_\nu D^\alpha_t \tilde{u}(t)\, D^\alpha_t \tilde{u}(t)\, ds\\
=& -\int_\Omega (D^\alpha_t)^2\tilde{u}(t) \,D^\alpha_t \tilde{u}(t)\, dx -\int_\Omega q\,(D^\alpha_t \tilde{u}(t))^2\, dx -\int_{\partial\Omega}\gamma (D^\alpha_t \tilde{u}(t))^2 ds\,.
\end{aligned}
\]
With \eqref{eqn:PF}, \eqref{eqn:boundsaq}, \eqref{eqn:ulaulqvarrho} we obtain 
\begin{equation}\label{eqn:estnabla}
\ul{a}\|\nabla D^\alpha_t \tilde{u}(t)\|_{\Ltwo}^2\leq \|(D^\alpha_t)^2\tilde{u}(t) \|_{\Ltwo} \|D^\alpha_t \tilde{u}(t)\|_{\Ltwo}\,.
\end{equation}
Moreover, for the $\triangle D^\alpha_t \tilde{u}(t)$ term in \eqref{eqn:estrhs}, we use the fact that for any $w\in H^2(\Omega)$ 
\[
\begin{aligned}
2\ul{a}\|\triangle w\|_{\Ltwo}
&\leq\|\tilde{a}\triangle w\|_{\Ltwo}
=\|\nabla\cdot(\tilde{a}\nabla w)-\nabla \tilde{a}\cdot\nabla w\|_{\Ltwo}\\
&\leq \|\nabla\cdot(\tilde{a}\nabla w)\|_{\Ltwo}+\|\tilde{a}'\|_{\Ltwo}\|\nabla w\|_{\Linf}\\
&\leq \|\nabla\cdot(\tilde{a}\nabla w)\|_{\Ltwo}
+\|\tilde{a}'\|_{\Ltwo}C_{H^1\to L^\infty}^\Omega \|(-\triangle+\mbox{id})^{-1}\|_{L^2(\Omega)\to H^2(\Omega)}(\|\triangle w\|_{\Ltwo}+\|w\|_{\Ltwo})
\end{aligned}
\]
hence by our assumption \eqref{eqn:smallnessaprime} on the
smallness of $\tilde{a}'$,
\begin{equation}\label{eqn:estw}
\|\triangle w\|_{\Ltwo}\leq \frac{1}{\ul{a}} \|\nabla\cdot(\tilde{a}\nabla w)\|_{\Ltwo} + \|w\|_{\Ltwo}\,,
\end{equation}
which we apply with $w=D^\alpha_t \tilde{u}(t)$.
In here, by \eqref{eqn:ztil},
\begin{equation}\label{eqn:estnabla2}
\|\nabla\cdot(\tilde{a}\nabla D^\alpha_t \tilde{u}(t))\|_{\Ltwo}
=\|(D^\alpha_t)^2\tilde{u}(t) +\tilde{q}D^\alpha_t \tilde{u}(t)\|_{\Ltwo}
\,,
\end{equation}
where
\begin{equation}\label{eqn:estnabla3}
\begin{aligned}
\|\tilde{q}D^\alpha_t \tilde{u}(t)\|_{\Ltwo}
&\leq\|\tilde{q}\|_{\Ltwo}C_{H^1\to L^\infty}^\Omega
\Bigl(\|\nabla D^\alpha_t \tilde{u}(t)\|_{\Ltwo}+\|D^\alpha_t \tilde{u}(t)\|_{\Ltwo}\Bigr)\\
&\leq\|\tilde{q}\|_{\Ltwo}C_{H^1\to L^\infty}^\Omega
\Bigl(\tfrac{1}{\ul{a}} \|(D^\alpha_t)^2\tilde{u}(t) \|_{\Ltwo}^{1/2} \|D^\alpha_t \tilde{u}(t)\|_{\Ltwo}^{1/2}+\|D^\alpha_t \tilde{u}(t)\|_{\Ltwo}\Bigr)\,,
\end{aligned}
\end{equation}
so that all spatial derivatives of $D^\alpha_t \tilde{u}$ that are needed for further estimating \eqref{eqn:estrhs} can be expressed via $D^\alpha_t \tilde{u}$ and $(D^\alpha_t)^2\tilde{u}$.

For these, we get from 
\[
D^\alpha_t \tilde{u}= -\widetilde{\mathbb{L}}\tilde{u}\,, \qquad 
(D^\alpha_t)^2\tilde{u}=\widetilde{\mathbb{L}}^2\tilde{u}\,,
\] 
that
\[
\begin{aligned}
D^\alpha_t \tilde{u}(x,t)=&-\sum_{j=1}^\infty \lamtil_j E_{\alpha,1}(-\lamtil_j t^\alpha)(u_0,\phitil_j) \phitil_j(x)
\\
(D^\alpha_t)^2\tilde{u} (x,t)=&\sum_{j=1}^\infty \lamtil_j^2 E_{\alpha,1}(-\lamtil_j t^\alpha)(u_0,\phitil_j) \phitil_j(x)\,.
\end{aligned}
\]
Using this and \eqref{eqn:estnabla}, \eqref{eqn:estw}, \eqref{eqn:estnabla2}, \eqref{eqn:estnabla3} in \eqref{eqn:estrhs}, we get \eqref{eqn:estrhs1}.
\end{proof}

\begin{lemma}
The function $\hat{u}$ defined by \eqref{eqn:repr_uhat}, with $\tilde{u}$ according to Lemma \oldref{lem:util} satisfies the estimate
\begin{equation}\label{eqn:estuhat2}
\|D^\alpha_t \hat{u}(T)\|_{\Ltwo}\leq
C_3 \max\{E_{\alpha,1}(-\lambda_1 T^\alpha),\Phi(T)\}
\Bigl(\|da'\|_{\Ltwo}+\|da\|_{\Linf}+\|dq\|_{\Ltwo}\Bigr)\,,
\end{equation}
with a constant $C_3$ depending only on $\Omega$, $\frac{1}{\ul{a}}$ and $\|\widetilde{\mathbb{L}}u_0\|_{\Ltwo}$ and 
\begin{equation}\label{eqn:Phi}
\Phi(T)=\sup_{\lambda\geq\lambda_1} \sup_{\mu\geq\tilde{\lambda}_1}
\int_0^T s^{\alpha-1}E_{\alpha,\alpha}(-\lambda s^\alpha)\max\{1,\mu\} E_{\alpha,1}(-\mu (T-s)^\alpha)\,ds
\end{equation}
\end{lemma}

\begin{proof}
Using the fact that $\hat{z}:=D^\alpha_t \hat{u}$ solves 
\begin{equation}\label{eqn:z}
D^\alpha_t \hat{z}-\nabla(a\nabla \hat{z})+q\hat{z} = \nabla(da\nabla D^\alpha_t\tilde{u})-dq\,D^\alpha_t\tilde{u} \quad t\in(0,T)\end{equation}
with initial conditions $\hat{z}(0)=D^\alpha_t\hat{u}(0)=\nabla(da\nabla u_0)-dq\,u_0$,
we get 
\[
\begin{aligned}
D^\alpha_t \hat{u}(x,T)=&\sum_{j=1}^\infty \Bigl\{E_{\alpha,1}(-\lambda_j T^\alpha)\langle\nabla\cdot(da\nabla u_0)-dq\,u_0,\varphi_j\rangle\\
&+\int_0^T s^{\alpha-1}E_{\alpha,\alpha}(-\lambda_j s^\alpha)\langle\nabla\cdot(da\nabla D^\alpha_t \tilde{u}(T-s))-dq\,D^\alpha_t \tilde{u}(T-s),\varphi_j\rangle\, ds\Bigr\} \varphi_j(x)\\
\end{aligned}
\]

Take the $L^2$ norm of the above expression for $D^\alpha_t \hat{u}$ to get, using an inequality of the form
\[
\sum_{j=1}^\infty\Bigl(\int_0^T a_j(s) b_j(T-s)\, ds\Bigr)^2
\leq \sup_{j\in\mathbb{N}}\Bigl(\int_0^T |a_j(s)|\, ds\Bigr)^2
\sum_{j=1}^\infty\Bigl(\sup_{t\in(0,T)} b_j(t)\Bigr)^2\,,
\]
that 
\begin{equation}\label{eqn:estuhat}
\begin{aligned}
&\|D^\alpha_t \hat{u}(T)\|_{\Ltwo}\\
&\leq
\Bigl(\sum_{j=1}^\infty E_{\alpha,1}(-\lambda_j T^\alpha)^2\langle\nabla\cdot(da\nabla u_0)-dq\,u_0,\varphi_j\rangle^2\Bigr)^{1/2}\\
&\qquad+\Bigl(\sum_{j=1}^\infty\Bigl(\int_0^T s^{\alpha-1}E_{\alpha,\alpha}(-\lambda_j s^\alpha)\langle\nabla\cdot(da\nabla D^\alpha_t \tilde{u}(T-s))-dq\,D^\alpha_t \tilde{u}(T-s),\varphi_j\rangle\, ds\Bigr)^2\Bigr)^{1/2}\\
&\leq E_{\alpha,1}(-\lambda_1 T^\alpha) \|\nabla\cdot(da\nabla u_0)-dq\,u_0\|_{\Ltwo}
\;+\;\Bigl[\sup_{j\in\mathbb{N}}\int_0^T s^{\alpha-1}E_{\alpha,\alpha}(-\lambda_j s^\alpha) e(\lamtil_1,T-s)\,ds \Bigr]\\
&\qquad\qquad\times\;\Bigl(\sum_{j=1}^\infty \sup_{t\in(0,T)} e(\lamtil_1,t)^{-1}\langle\nabla\cdot(da\nabla D^\alpha_t \tilde{u}(t))-dq\,D^\alpha_t \tilde{u}(t),\varphi_j\rangle^2\Bigr)^{1/2}\\
&\leq E_{\alpha,1}(-\lambda_1 T^\alpha) \Bigl(\|da'\|_{\Ltwo} \|u_0'\|_{\Linf} + \|da\|_{\Linf} \|u_0''\|_{\Ltwo}+ \|dq\|_{\Ltwo} \|u_0\|_{\Linf}\Bigr)\\
&\qquad+\sup_{\lambda\geq\lambda_1} \sup_{\mu\geq\tilde{\lambda}_1}
\int_0^T s^{\alpha-1}E_{\alpha,\alpha}(-\lambda_j s^\alpha)\max\{1,\mu\} E_{\alpha,1}(-\mu (T-s)^\alpha)\,ds\\
&\qquad\qquad \sup_{t\in(0,T)} e(\lamtil_1,t)^{-1}\|\nabla\cdot(da\nabla D^\alpha_t \tilde{u}(t))-dq\,D^\alpha_t \tilde{u}(t)\|_{\Ltwo}\,.
\end{aligned}
\end{equation}
Thus from \eqref{eqn:estrhs1} we get \eqref{eqn:estuhat2} with \eqref{eqn:Phi}.
\end{proof}

It is therefore crucial for contractivity to prove that $\Phi(T)$ as defined in \eqref{eqn:Phi} tends to zero for increasing $T$.

\begin{lemma} \label{lem:Phi}
For $\Phi(T)$ according to \eqref{eqn:Phi} we have $\Phi(T)\to0$ as $T\to\infty$.
\end{lemma}

\begin{proof}
Using the identities
\[
\lambda s^{\alpha-1}E_{\alpha,\alpha}(-\lambda s^\alpha)=\frac{d}{ds} E_{\alpha,1}(-\lambda s^\alpha)\,, \quad 
E_{\alpha,1}(0)=1\ \quad E_{\alpha,\alpha}(0)=\frac{1}{\Gamma(\alpha)}\,,
\]
and the bound
\[
E_{\alpha,1}(-x)\leq \frac{1}{1+\Gamma(1+\alpha)^{-1}x}
\]
that hold for every $\alpha\in(0,1)$ and all $s\in\mathbb{R}$ and all $x\in\mathbb{R}^+$,
as well as the complete monotonicity of the function $x\mapsto E_{\alpha,1}(-x)$ on $\mathbb{R}^+$,
see, e.g., \cite{Djrbashian:1993,JinRundell:2015}, we can estimate as follows 
\[
\begin{aligned}
&\int_0^T s^{\alpha-1}E_{\alpha,\alpha}(-\lambda s^\alpha) \mu E_{\alpha,1}(-\mu (T-s)^\alpha)\,ds\\
&\leq \Gamma(1+\alpha)\Bigl(\int_0^{\frac{T}{2}} s^{\alpha-1}E_{\alpha,\alpha}(-\lambda s^\alpha)  (T-s)^{-\alpha}\,ds
+\int_{\frac{T}{2}}^T s^{\alpha-1}E_{\alpha,\alpha}(-\lambda s^\alpha)  (T-s)^{-\alpha}\,ds\Bigr)
\\
&=\Gamma(1+\alpha)\Bigl(\tfrac{\alpha}{\lambda}\int_0^{\tfrac{T}{2}} E_{\alpha,1}(-\lambda s^\alpha)  (T-s)^{-\alpha-1}\,ds
- \tfrac{\alpha}{\lambda}T^{-\alpha}
+\tfrac{\alpha}{\lambda}E_{\alpha,1}(-\lambda (\tfrac{T}{2})^\alpha) (\tfrac{T}{2})^{-\alpha}\\
&\qquad\qquad\qquad + E_{\alpha,\alpha}(-\lambda (\tfrac{T}{2})^\alpha)\int_{\tfrac{T}{2}}^T  s^{\alpha-1} (T-s)^{-\alpha}\,ds
\Bigr)
\\
&\leq\Gamma(1+\alpha)\Bigl(\tfrac{\alpha}{\lambda}\int_0^{\tfrac{T}{2}} (T-s)^{-\alpha-1}\,ds
+ \tfrac{\alpha}{\lambda}(\tfrac{T}{2})^{-\alpha} + E_{\alpha,\alpha}(-\lambda (\tfrac{T}{2})^\alpha)\int_{\tfrac{1}{2}}^1  r^{\alpha-1} (1-r)^{-\alpha}\,ds\Bigr)\,,
\end{aligned}
\]
where we have used the substitution $s=Tr$.
Hence we get 
\[
\begin{aligned}
\Phi(T)&\leq \frac{\Gamma(1+\alpha)}{\min\{1,\tilde{\lambda}_1\}}\Bigl(
\tfrac{1+\alpha}{\lambda_1}(\tfrac{T}{2})^{-\alpha}
+ E_{\alpha,\alpha}(-\lambda_1 (\tfrac{T}{2})^\alpha)\int_{\tfrac{1}{2}}^1  r^{\alpha-1} (1-r)^{-\alpha}\,ds \Bigr)\\ 
&\to0 \mbox{ as }T\to\infty\,.
\end{aligned}
\]
\end{proof}

This together with \eqref{eqn:estdadq}, \eqref{eqn:estuhat2} yields contractivity of $\mathbb{T}$ for sufficiently large $T$ on the closed convex set 
\begin{equation}\label{eqn:B}
\begin{aligned}
B=\Bigl\{(a,q)\in &H^1(\Omega)\times L^2(\Omega)\, : \, a(0)=a_0(0)\,, \quad a(x)\geq\ul{a}\, \quad \|\tilde{q}-\ul{q}\|_{\Ltwo}\leq \frac{\varrho}{(C_{H^1\to L^\infty}^\Omega)^2}\,, \\ &
\|a'\|_{\Ltwo}\leq \frac{\ul{a}}{C_{H^1\to L^\infty}^\Omega \|(-\triangle+\mbox{id})^{-1}\|_{L^2(\Omega)\to H^2(\Omega)}}
\Bigr\}\,.
\end{aligned}
\end{equation}
with the norm 
\begin{equation}\label{eqn:norm}
|||(a,q)|||:=\|a'\|_{\Ltwo}+\|a\|_{\Linf}+\|q\|_{\Ltwo}\,.
\end{equation}
In the definition of $B$, we assume that $\ul{a}>0$, $\ul{q}\in\mathbb{R}$, $\rho>0$ satisfy 
\eqref{eqn:ulaulqvarrho} and 
\begin{equation}\label{eqn:ulaulqvarrho2}
\varrho<\min\{\tfrac{\ul{a}}{C_{PF\Omega}},\tfrac{1}{C_{PF\partial\Omega}}\}+\ul{q}\,.
\end{equation}
Boundedness away from zero of $\lambda_1=\lambda_1(a,q)$, uniformly for $(a,q)\in B$ follows from the identity 
\[
\begin{aligned}
\lambda_j \|\varphi_j\|_{\Ltwo}^2 
=&\int_\Omega \Bigl(-\nabla\cdot(a\nabla\varphi_j) + q \varphi_j^2\Bigr)\, dx \\
=&\int_\Omega \Bigl(a|\nabla\varphi_j|^2+q \varphi_j^2\Bigr)\, dx +\int_{\partial\Omega}\gamma \varphi_j^2 \, ds\\
\geq&2\ul{a}\|\nabla \varphi_j\|_{\Ltwo}^2+\ul{q}\|\varphi_j\|_{\Ltwo}^2 +\int_{\partial\Omega}\gamma \varphi_j^2 \, ds - \varrho \Bigl(\|\nabla \varphi_j\|_{\Ltwo}^2+\|\varphi_j\|_{\Ltwo}^2\Bigr)
\end{aligned}
\]
hence
\[
(\lambda_j+\varrho-\ul{q}) \|\varphi_j\|_{\Ltwo}^2\geq (2\ul{a}-\varrho)\|\nabla \varphi_j\|_{\Ltwo}^2
+\int_{\partial\Omega}\gamma \varphi_j^2 \, ds
\geq \min\{\tfrac{\ul{a}}{C_{PF\Omega}},\tfrac{1}{C_{PF\partial\Omega}}\} \|\varphi_j\|_{\Ltwo}^2
\]
which together with the constraints \eqref{eqn:ulaulqvarrho} and \eqref{eqn:ulaulqvarrho2} on $\ul{a}$, $\ul{q}$ and $\varrho$ yields 
\[
\lambda_j\geq \ul{\lambda} := \min\{\tfrac{\ul{a}}{C_{PF\Omega}},\tfrac{1}{C_{PF\partial\Omega}}\}-\varrho+\ul{q} >0\,.
\]
Thus \eqref{eqn:estuhat2} remains valid with $\min\{\lambda_1,\lamtil_1\}$ replaced by $\ul{\lambda}$ and we get contractivity with a uniform constant on $B$, provided $T$ is sufficiently large.

\begin{theorem}\label{theorem:contractivity}
Let $\Omega=(0,L)\subseteq\mathbb{R}$, $r_u=r_v=0$, $\alpha\in (0,1)$ and let $g_u,g_v\in H^2(\Omega)$ satisfy $|\frac{1}{W(x)}|\leq C_0$ for $W$ defined in \eqref{eqn:W}.
\\
Then for $T>0$ sufficiently large, the operator $\mathbb{T}$ defined by $\mathbb{T}(a,q)=(a^+,q^+)$ solving \eqref{eqn:aplusqplus} with $a^+(0)=a_0$ or $a^+(L)=a_L$, is a contraction on the set \eqref{eqn:B} with respect to the norm defined by \eqref{eqn:norm}.
\end{theorem}

\begin{remark}
Under the conditions of Theorem \oldref{theorem:contractivity} we have convergence of the fixed point iteration defined by $(a_{k+1},q_{k+1})=\mathbb{T}(a_k,q_k)$ for any starting value in $B$. Moreover, we have uniqueness of a solution to the inverse problem on $B$.
\end{remark}

\begin{remark}\label{rem:noise}
In the realistic setting of noisy data $g_u^\delta\approx g_u$,
$g_v^\delta\approx g_v$, $g_u^\delta, g_v^\delta \in L^2(\Omega)$ with
$\|g_u-u(\cdot,T;a_{act},q_{act})\|_{\Ltwo}\leq \delta$,
$\|g_v-v(\cdot,T;a_{act},q_{act})\|_{\Ltwo}\leq \delta$, we first of all
filter the data to get approximations
$\tilde{g}_u^\delta, \tilde{g}_v^\delta \in H^2(\Omega)$ with
$$
\|\tilde{g}_u-u(\cdot,T;a_{act},q_{act})\|_{H^2(\Omega)}\leq \tilde{\delta},
\quad
\|\tilde{g}_v-v(\cdot,T;a_{act},q_{act})\|_{H^2(\Omega)}\leq \tilde{\delta},
$$
where $\tilde{\delta}=O(\psi(\delta))$ with an index function
$\psi$\footnote{i.e., a continuous non-decreasing function
$\psi : (0, \infty) \to\mathbb{R}^+$ with $\lim_{x\to0}\psi(x)=0$}
depending on the smoothness of both $v(\cdot,T;a_{act},q_{act})$,
$u(\cdot,T;a_{act},q_{act})$.
Then we perform the fixed point iteration up to a stopping index
$k_*(\tilde{\delta})\sim\log(1/\tilde{\delta})$, see, e.g.,
\cite[Section 3.5]{KaltenbacherRundell:2019c} to obtain a convergence rate
$|||(a_{k_*(\tilde{\delta})},q_{k_*(\tilde{\delta})})-(a_{act},q_{act})|||=O(\tilde{\delta})$.   
\end{remark}

\subsection{Contractivity in the potential only case}\label{sec:potentialonly}
Contractivity, hence convergence of the fixed point iteration and uniqueness can be proven also in higher space dimensions $\Omega\subseteq\mathbb{R}^3$ in case the diffusion coefficient $a$ is known, see also \cite{KaltenbacherRundell:2019b}. This is due to the fact that the regularity requirements on $D^\alpha_t u(T)$ can be weakened to just an $L^\infty$ estimate as we only recover an $L^2$ coefficient $q$, while we needed to estimating $D^\alpha_t \hat{u}(T)$ in $H^2(\Omega)$ for obtaining also $a$ in $H^1(\Omega)\cap W^{1,\infty}(\Omega)$.
Since the proof in fact very much follows the line of the previous section, we will here only show the key steps.

To this end, we consider the problem of identifying $q(x)$ in 
\begin{equation}\label{eqn:u_qonly}
\begin{aligned}
&D^\alpha_t u-\triangle u+qu=0 \quad t\in(0,T)\,, \quad u(0)=u_0
\end{aligned}
\end{equation}
with homogeneous impedance boundary conditions 
\[
\partial_\nu u+\gamma u =0
\]
from observations $g(x)=u(x,T)$.
Note that $-\triangle$ can be replaced by an arbitrary second order elliptic differential operator with sufficiently smooth coefficients. 

Define a fixed point operator $\mathbb{T}$ by 
\begin{equation}\label{eqn:T_qonly}
\mathbb{T}(q)=q^+=\frac{\triangle g-D^\alpha_t u(\cdot,T;q)}{g}, 
\mbox{ where  $u(x,t):= u(x,t;q)$ solves \eqref{eqn:u_qonly}.}
\end{equation}
For two different potentials $q$, $\tilde{q}$ (with corresponding solutions $u$, $\tilde{u}:= u(x,t;\tilde{q})$,  the difference $dq^+=\mathbb{T}(q)-\mathbb{T}(\tilde{q})=\frac{-D^\alpha_t u(T)+D^\alpha_t \tilde{u}(T)}{g}:=-\frac{D^\alpha_t \hat{u}(T)}{g}$ where $\hat{u}$, solves 
\begin{equation}\label{eqn:uhat}
\begin{aligned}
&D^\alpha_t \hat{u}-\triangle \hat{u}+q\hat{u} = -dq\,\tilde{u} \quad t\in(0,T)\,, \quad \hat{u}(0)=0
\end{aligned}
\end{equation}
with $dq=q-\tilde{q}$.

We use eigensystems $(\lambda_j, \varphi_j)$ $(\lamtil_j,\phitil_j)$ of the operators defined by $\mathbb{L}w=-\triangle w+q\,w$ and $\widetilde{\mathbb{L}}w=-\triangle w+\tilde{q}w$, to obtain the representations
\[
\begin{aligned}
\hat{u}(x,t)=&\sum_{j=1}^\infty \int_0^t s^{\alpha-1}E_{\alpha,\alpha}(-\lambda_j s)\langle-dq\,\tilde{u}(t-s),\varphi_j\rangle\, ds \varphi_j(x)\\
\tilde{u}(x,t)=&\sum_{j=1}^\infty E_{\alpha,1}(-\lambda_j t)(u_0,\phitil_j) \phitil_j(x)\\
D^\alpha_t \hat{u}(x,T)=&\sum_{j=1}^\infty \Bigl\{E_{\alpha,1}(-\lambda_j t)\langle-dq\,u_0,\varphi_j\rangle+\int_0^T\!\! s^{\alpha-1}E_{\alpha,\alpha}(-\lambda_j s)\langle-dq\,D^\alpha_t \tilde{u}(T\!-\!s),\varphi_j\rangle\, ds\Bigr\} \varphi_j(x)\\
D^\alpha_t \tilde{u}(x,t)=&-\sum_{j=1}^\infty \lamtil_j E_{\alpha,1}(-\lambda_j t)\langle u_0,\phitil_j\rangle \phitil_j(x)\,.
\end{aligned}
\]
Here
\[
\|D^\alpha_t \tilde{u}(t)\|_{\Linf}
\leq C_{\dot{H}^\sigma\to L^\infty}^\Omega \|D^\alpha_t \tilde{u}(t)\|_{\dot{H}^\sigma(\Omega)}
\leq C_0 \sup_{\mu\geq\lamtil_1} \mu^{\sigma/2}E_{\alpha,1}(-\mu t^\alpha)
\]
for $C_0=C_{\dot{H}^\sigma\to L^\infty}^\Omega \|\widetilde{\mathbb{L}} u_0\|_{\Ltwo}$. 
Hence
\[
\begin{aligned}
\|D_t^\alpha \hat{u}(T)\|_{\Ltwo}\leq
&\left(\sum_{j=1}^\infty E_{\alpha,1}(-\lambda_j T^\alpha)\langle dq\,u_0,\varphi_j\rangle^2\right)^{1/2}\\
&+\left(\sum_{j=1}^\infty\left(\int_0^T\! s^{\alpha-1}E_{\alpha,\alpha}(-\lambda_j s^\alpha)\langle dq\,D^\alpha_t \tilde{u}(T-s),\varphi_j\rangle\, ds\right)^2\right)^{1/2}\\
\leq &E_{\alpha,1}(-\lambda_1 T^\alpha) \|dq\,u_0\|_{\Ltwo}\\
&+C_0\|dq\|_{\Ltwo}
\sup_{\lambda\geq\lambda_1} \sup_{\mu\geq\lamtil_1}
\int_0^T \!s^{\alpha-1}E_{\alpha,\alpha}(-\lambda s^\alpha) \, \mu^{\sigma/2}E_{\alpha,1}(-\mu (T-s)^\alpha)\, ds\,\,.
\end{aligned}
\]
which for $\sigma\leq2$, by Lemma \oldref{lem:Phi},
tends to zero as $T\to\infty$.
Due to the fact that $H^\sigma(\Omega)$ continuously embeds into
$L^\infty(\Omega)$ for a bounded domain $\Omega\subseteq\mathbb{R}^3$, this is works in up to three space dimensions and gives contractivity for $T$ large enough.

\begin{theorem}\label{theorem:contractivity_qonly}
Let $\Omega\subseteq\mathbb{R}^3$ be a bounded $C^{1,1}$ domain, $r=0$, $\alpha\in (0,1)$ and let $g\in H^2(\Omega)$ satisfy $|\frac{1}{g(x)}|\leq C_0$.
Then for $T>0$ sufficiently large, the operator $\mathbb{T}$ defined by \eqref{eqn:T_qonly} is a contraction with respect to the $L^2$ norm.
\end{theorem}

\input recon_figs
\section{Reconstructions}\label{sect:recons}

We will show the results of numerical experiments with the three versions
of the basic iterative scheme:  compute $a$, $q$ in parallel;
eliminate $q$ recover $a$; eliminate $a$ recover $q$,

In the reconstructions to be shown we used the initial values $\alpha=1$, $T=0.5$
and a noise level (uniformly distributed) of 1\% as a basis for discussion.
At the end of the section we will indicate the effective dependence of the
reconstruction process on these quantities.
The reconstructions we show will be in $\mathbb{R}$ as the
graphical illustration is then more transparent and there is little to be
gained technically or visually from higher dimensions.

We will also take the following actual functions to be reconstructed as
$$
a_{\rm act}(x) = 1 + 4x^2(1-x) + 0.5\sin(4\pi x)
\quad\qquad
q_{\rm act}(x) = 8x\,e^{-3x}
$$
As data we took two differing initial values $u_0(x)$ and $v_0(x)$
and as boundary conditions we used (nonhomogeneous)
Dirichlet at the left endpoint and Neumann at the right; typically
different for each of $u$ and $v$.

One such data set is shown in Figure~\oldref{fig:data_values}
\begin{figure}[ht]
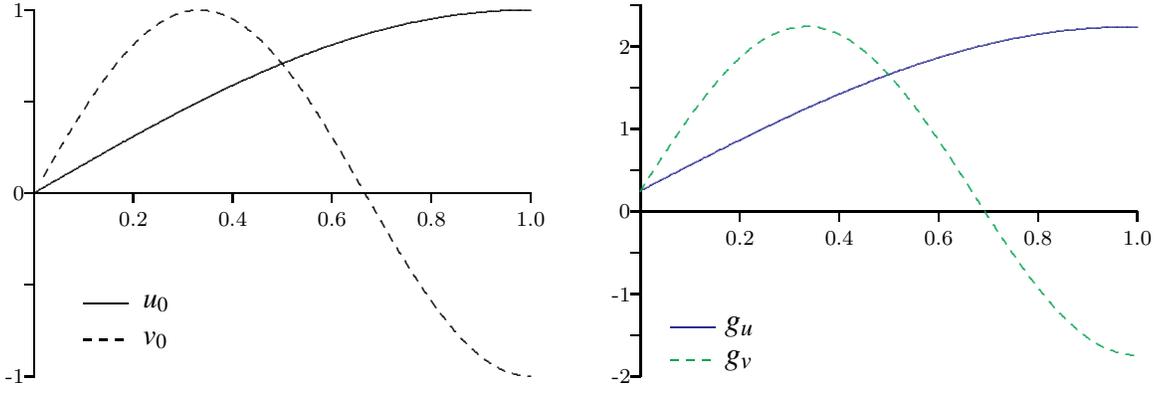

\hbox to\hsize{\copy\figuresix\hss \copy\figureseven}
\caption{\small {\bf Initial values $u_0(x)$, $v_0(x)$ and data $g_u=u(x,T)$, $g_v=v(x,T)$}}
\label{fig:data_values}
\end{figure}

\subsection{Performance of the three schemes}

In the parallel scheme with actual values,
since we make no constraints on the form of the unknown functions
other than sufficient regularity,
we do not choose a basis with in-built restrictions as would be
obtained from an eigenfunction expansion.
Instead we used a radial basis of shifted Gaussian functions
$b_j(x) := e^{-(x-x_j)^2/\sigma}$ centered at nodal points  $\{x_j\}$
and with width specified by the parameter $\sigma$.
The linear operator $\mathbb{M}$, in equation~\eqref{eqn:uv_sys2}
then takes the matrix form
$$
\mathbb{M} =
\begin{bmatrix}
A_1 & Q_1 \\
A_2 & Q_2
\end{bmatrix}
$$
where $A_1$ denotes the representation of $a(x)$ using the values $g_u$
and $Q_2$ the representation of $q(x)$ using the values $g_v$.

The sequential schemes are based on eliminating one of $a(x)$ or $q(x)$ and
having $\mathbb{M}$ represented through pointwise values of the functions
$W$ and $W'$.

The singular values of the component matrices $A_1$ and $Q_1$ are shown in
the leftmost figure in Figure~\oldref{fig:sv_AQ_WWp}
and the functions $W$ and $W'$ in the rightmost figure .

\begin{figure}[ht]
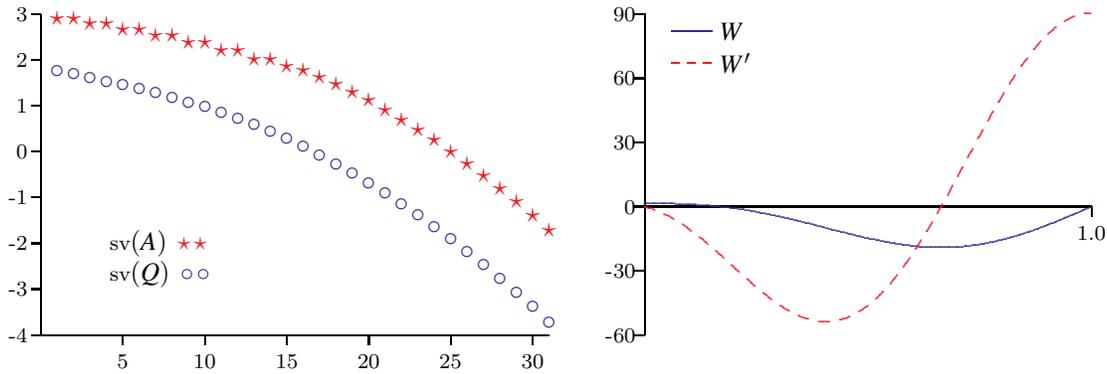

\hbox to \hsize{\hss\copy\figureone\hss\hss\copy\figureoneb\hss}
\caption{\small {\bf  Left:  Singular values of the matrices $A$, $Q$:\qquad
Right: The functions $W$ and $W'$}}
\label{fig:sv_AQ_WWp}
\end{figure}

Even before seeing the resulting reconstructions it is clear the far superior
conditioning of the $A$ matrix over that of $Q$ -- a factor of over 10 in
the larger singular values and of 100 in the lower ones --
is going to significantly favour the reconstruction of the $a(x)$ coefficient.
This is also borne out from the rightmost figure here:
while the values of the function $W$ are modest,
there is a much larger range in the values of $W'$.
Note that from equation~\eqref{eqn:a_uncoupled}
the reconstruction of $a(x)$ requires only $W$,
but that of \eqref{eqn:q_uncoupled} which updates $q(x)$ requires
both $W$ and $W'$.
Indeed this turns out to be the situation in both cases.

Using the parallel scheme reconstructions of $a$ and $q$
from two initial/boundary values under $1\%$ random uniform noise are shown in
Figure~\oldref{fig:recon_a_q_parallel}
The initial approximations  were $a(x)=1$ and $q(x)=0$.
The first iteration resulted in an already near perfect reconstruction of
the $a(x)$ coefficient but that of $q(x)$ lagged significantly behind and
in the end the error in the data measurements were predominately in this
coefficient.

\begin{figure}[ht]
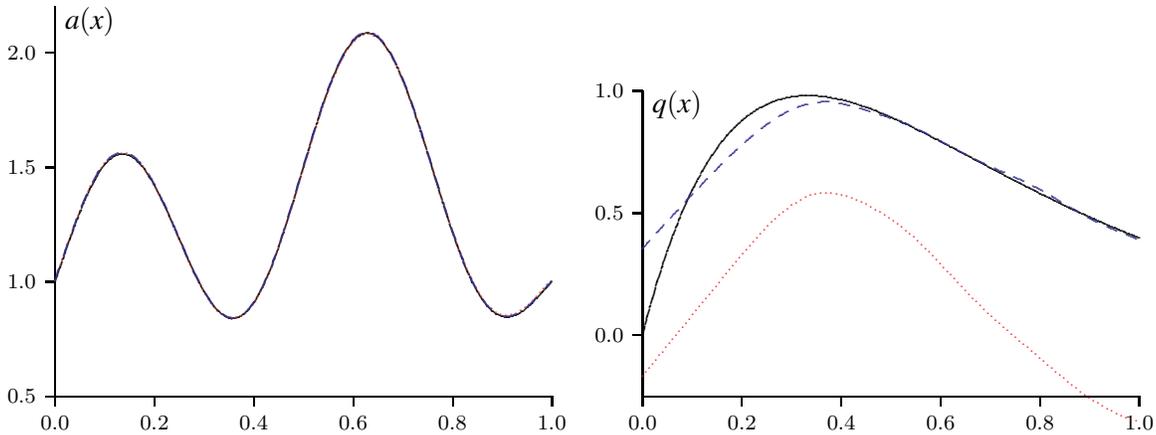

\hbox to\hsize{\copy\figuretwo\hss \copy\figurethree}
\caption{\small {\bf  Reconstructions of $a$ and $q$ using parallel algorithm}}
\label{fig:recon_a_q_parallel}
\end{figure}

In the eliminate $q$ and update $a$ strategy some care must be taken here
as the function $W(x)$ has two zeros; an interior one around $\tilde x = 0.17$
and at the endpoint $x=1$.
Thus a straightforward division by $W$ to recover $a$ from
\eqref{eqn:a_uncoupled} isn't possible.
In theory the right-hand side $\Phi(x)=\int_0^x\phi(s)\,ds$
 term should also vanish at these points including the interior one $\tilde x$,
but with data noise this isn't going to be the case.
In practice we found the following device worked very well.
Remove a small segment of several grid points from an interval
$I_{\tilde x} =(\tilde x-\delta,\tilde x+\delta)$, set $\Phi(\tilde x) = 0$ and then
fill in the interval by interpolation using say a smoothing spline
with the level of smoothing chosen depending on the assumed noise in $\phi$.
Then a straight division recovers $a(x)$.
A new direct solve then is used to recover the next iteration of $q(x)$.
As Figure~\oldref{fig:a_q_iter2} shows the results are comparable to the previous
reconstruction.

It is worth observing from Figures~\oldref{fig:data_values} and
\oldref{fig:a_q_iter2},
that those regions where the reconstructions of $q$ are poorest coincide with
regions of smaller values of $g_u(x)=u(x,T)$ and $g_v(x)=v(x,T)$,
namely near the left hand point of the interval.
This is in keeping with the fact that both $W$ and $W'$ are smaller in
magnitude at these points.

\begin{figure}[ht]
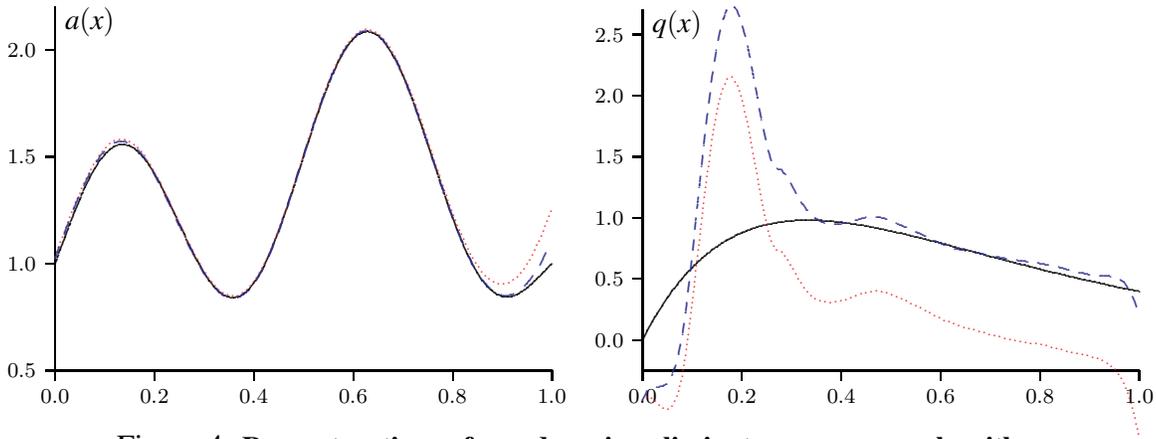

\hbox to\hsize{\copy\figurefour\hss \copy\figurefive}
\caption{\small {\bf Reconstructions of $a$ and $q$ using eliminate $a$,
recover $q$ algorithm}}
\label{fig:a_q_iter2}
\end{figure}

We do not show a reconstruction for the version based on the update
equation~\eqref{eqn:q_uncoupled}. While a somewhat satisfactory reconstruction
of the coefficient $a(x)$ was obtained, the scheme failed to converge for
the coefficient $q(x)$.
This fact alone shows how dominant a role the diffusion coefficient plays
in the process at the expense of the much weaker potential term.
It is only under significantly less data noise
that an effective reconstruction of the latter was possible.
The relative rates of convergence of all three versions is shown in
the displayed table.

\begin{table}[ht]
\footnotesize
\renewcommand{\arraystretch}{1.1}
\centering
\begin{tabular}{| l c c c c c c |}
\hline
\hline
Iteration & 1 & 2 & 3 & 4 & 5 & 6\\
\hline
\hline
{\small{\bf  Parallel scheme}}& & & & & & \\
\qquad$\|a_n-a_{\rm act}\|_\infty$ & 0.0046 & 0.0071 & 0.0052 & 0.0042 & 0.0038 & - \\
\qquad$\|q_n-q_{\rm act}\|_\infty$ & 0.7723 & 0.2845 & 0.3368 & 0.3561 & 0.3618 & - \\
\qquad$\|a_n-a_{\rm act}\|_2$ & 0.0039 & 0.0093 & 0.0055 & 0.0038 & 0.0031 & - \\
\qquad$\|q_n-q_{\rm act}\|_2$ & 0.7372 & 0.3109 & 0.1433 & 0.0980 & 0.0904 & - \\
\hline
{\small{\bf  Eliminate $q$ scheme}}& & & & & & \\
\qquad$\|a_n-a_{\rm act}\|_\infty$ & 0.0084 & 0.0099 & 0.0076 & 0.0065 & 0.0062 & 0.0061 \\
\qquad$\|q_n-q_{\rm act}\|_\infty$ & 0.8080 & 0.2957 & 0.3066 & 0.3258 & 0.3321 & 0.3343 \\
\qquad$\|a_n-a_{\rm act}\|_2$ & 0.0053 & 0.0092 & 0.0051 & 0.0036 & 0.0032 & 0.0030 \\
\qquad$\|q_n-q_{\rm act}\|_2$ & 0.8510 & 0.3320 & 0.1442 & 0.0961 & 0.0892 & 0.0886 \\
\hline
{\small{\bf  Eliminate $a$ scheme}}& & & & & & \\
\qquad$\|a_n-a_{\rm act}\|_\infty$ & 0.1254 & 0.0541 & 0.0518 & 0.0520 & 0.0521 & - \\
\qquad$\|q_n-q_{\rm act}\|_\infty$ & 1.3512 & 1.7705 & 1.8898 & 1.9318 & 1.9462 & - \\
\qquad$\|a_n-a_{\rm act}\|_2$ & 0.0347 & 0.0151 & 0.0123 & 0.0119 & 0.0118 & - \\
\qquad$\|q_n-q_{\rm act}\|_2$ & 0.9042 & 0.7562 & 0.7626 & 0.7788 & 0.7860 & - \\
\hline
\end{tabular}\label{table:norms}
\caption{\small {\bf Norm differences for versions of the iteration}}
\end{table}

The difference in the iteration counts in the table is
due to the use of a discrepancy principle as a stopping rule,
which terminates the iteration as soon as the residual drops below the noise
level. The discrepancy principle is well established as a regularization
parameter choice for ill-posed problems see, e.g.,
\cite{EnglHankeNeubauer:1996}, \cite{Morozov67}.
Here note that regularization is done by smoothing the data here
(cf. Remark \oldref{rem:noise}) so the stopping index does not act as a
regularization parameter; merely the discrepancy principle appears to find
a good final iterate.

\subsection{Changing some of the parameters}

The above reconstructions were set with the final time taken to be
$T=0.5$; the question is how the schemes would progress for different $T$
In particular, we are interested in this as one should expect the contraction
constant (if indeed there is a contraction) to be smaller with increasing $T$.

The answer is much as expected; the contraction constant varies with $T$,
and more generally, so does the strength of the nonlinear contribution
terms $u_t(x,T;a,q)$ and $v_t(x,T;a,q)$.
We would thus expect more iterations to be required as $T$ was reduced and
indeed this is the case.
As an illustration of this effect, Figure~\oldref{fig:a_T_parallel} shows
the first iterate of $a_1(x)$, for the values of $T=0.1$ and $T=0.05$.
Recall from Figure~\oldref{fig:recon_a_q_parallel} that even the first iteration
$a_1$ was sufficiently close to the actual when $T=0.5$ so that it was barely
indistinguishable from the actual.
We do not show the reconstructions of $q(x)$ here as in neither case of
$T=0.1$ nor of $T=0.05$ did these converge.
Indeed for these smaller value of $T$ the iterations oscillated widely
without any sense of convergence, it being quite clear that we were in a region
of non-convergence.

\begin{figure}[ht]
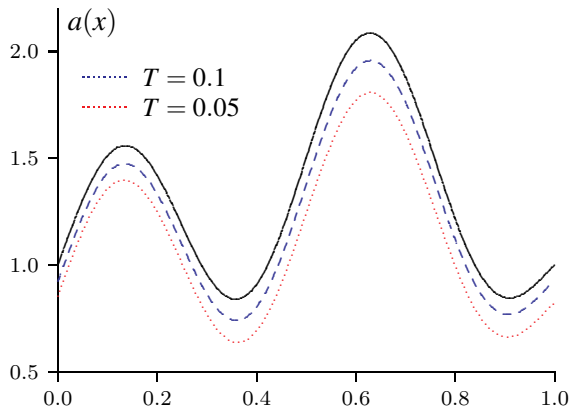

\hbox to\hsize{\hss\copy\figureeight\hss} 
\caption{\small Variation of the first iterate $a_1$ with $T$}
\label{fig:a_T_parallel}
\end{figure}

Since the coefficient $a(x)$ is not the determining factor in the
reconstruction process,
as a function of $\alpha$ the scheme reverts to our ability to determine $q$
exactly as in \cite{KaltenbacherRundell:2019b}.
Thus, as expected, the contraction constant increases with decreasing $\alpha$;
quite sharply at first.
Hence the convergence rate decreases with decreasing $\alpha$.
\medskip

The variation of the schemes with noise level is now predictable.
With an error much in excess of $1\%$ the schemes degrade rapidly.
Even for errors less than this value, the schemes based on pointwise evaluation
and requiring $W'$ did quite poorly, especially for smaller values of the
time measurement $T$ and the non-parabolic case.

\subsection{Measurements at two later times}

Finally, we indicate a little about the possibility of a single data run but
taking spatial measurements at two later times $t=T_1$ and $t=T_2$.
Some reflection will show the inherent difficulties here.
Unless there is some dynamic change in the solution profile for $T_1<t<T_2$
the likelihood is that
$g_1(x) = u(x,T_1)$ and $g_2(x) = u(x,T_2)$
will not differ sufficiently to allow effective division by the
``Wronskian'' function $W = g_1g_2' - g_1'g_2$.
Such a change could occur by an input of a large nonhomogeneous forcing function
or a change in the boundary conditions between these time values.

Consider the case when there are no such changes; we use the first solution
$u(x,t)$ above but take $g_1=u(x,T/2)$ and $g_2=u(x,T)$, with $T$ again the
value $T=0.5$.
The story here is told by looking at a plot of the resulting functions $W$
and $W'$ as shown in Figure~\oldref{fig:sv_W_Wp_just_u}
which highlights the weak linear independence of the two data values.

\begin{figure}[ht]
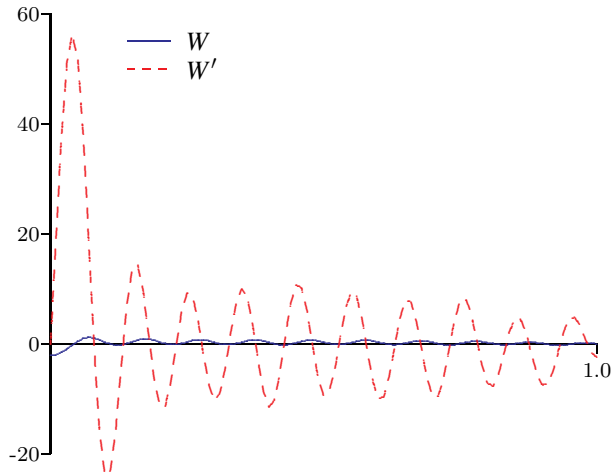

\hbox to \hsize{\hss\copy\figurenine\hss}
\caption{\small {\bf The functions $W$ and $W'$ for $g_v = u(x,\frac{T}{2})$}}
\label{fig:sv_W_Wp_just_u}
\end{figure}

This shows $W$ is very small in magnitude with several zeros;
$W'$ is in consequence highly oscillatory and relatively large in magnitude.
These are exactly the features one would try to avoid in making up a data set.

\subsection{Inclusion of a known nonlinearity}

In \cite{KaltenbacherRundell:2019b} it was shown that it was possible
to reconstruct the potential $q(x)$ occurring in the term $q(x)f(u)$
where $f(u)$ is assumed known.
Thus equation~\eqref{eqn:nonlin_pde} seeks to generalize this to also include
the determination of the conductivity $a(x)$.

Some of the methods described in section~\ref{subsec:iter} will no longer
quite work as, for example, the possibility to eliminate $q$
isn't directly feasible as in the case of $f(u)=u$.
However the scheme based on basis functions and determining $a$ and $q$ in
parallel goes through exactly as before since the operator $\mathbb{M}$
remains linear due to the fact that we are evaluating the known $f$ at
the data values $g_u$ and $g_v$.

Reconstructions for this case do depend quite strongly on $f$.
If $f(u)$ decays more rapidly than $u$ itself then it is quite likely
the recovery of $q(x)$ will be even more challenging.
On the other hand if $f(u)$ is substantially larger than $u$
(but assuming the resulting solutions of equation~\eqref{eqn:nonlin_pde}
remain bounded) then the term involving $q(x)$ plays a more substantial
role and the recovery of $q$ is much improved.

\section*{Acknowledgment}

\noindent
The work of the first author was supported by the Austrian Science Fund {\sc fwf}
under the grants I2271 and P30054.

\noindent
The work of the second author was supported 
in part by the
National Science Foundation through award {\sc dms}-1620138.


\end{document}

%% file: recon_figs.tex
\input colordvi
\input fonts
\input pictex
\font\smallsymbol = cmmi8
\newdimen\xfiglen \newdimen\yfiglen
\xfiglen=2.6 true in
\yfiglen=1.6 true in
\newbox\figurelegendone
\newbox\figurelegendtwo
\newbox\figurelegendthree
\newbox\figureone
\newbox\figureoneb
\newbox\figuretwo
\newbox\figurethree 
\newbox\figurefour
\newbox\figurefive
\newbox\figuresix
\newbox\figureseven
\newbox\figureeight
\newbox\figurenine
\newbox\figureten

\setbox\figurelegendone=\hbox{
\small
\beginpicture
  \setcoordinatesystem units <0.3\xfiglen,0.5\yfiglen> 
  \setplotarea x from 0 to 0.8, y from 0 to 0.4
\footnotesize
\linethickness=0.8pt
\eightrm
\setsolid
  \put {sv$(A)$ \Red{$\,\star\,\star$}} [l] at 0.3 0.3
  \put {sv$(Q)$ \Blue{$\,\circ\,\circ$}} [l] at 0.3 0.1
\endpicture
\relax
}
\setbox\figurelegendtwo=\hbox{
\small
\beginpicture
  \setcoordinatesystem units <0.3\xfiglen,0.45\yfiglen> 
  \setplotarea x from 0 to 0.8, y from 0 to 0.7
\ninerm
\linethickness=0.8pt
  \put {$\Red{\circ}$\quad Frozen Newton} [l] at 0 0.6
  \put {$\OliveGreen{\bullet}$\quad Newton} [l] at 0 0.4
  \put {$\Blue{\diamond}$\quad Frozen Halley} [l] at 0 0.2
  \put {$\Black{\star}$\quad Halley} [l] at 0 0.0
\endpicture
\relax
}
\setbox\figurelegendthree=\hbox{
\small
\beginpicture
  \setcoordinatesystem units <0.3\xfiglen,0.5\yfiglen> 
  \setplotarea x from 0 to 0.8, y from 0 to 0.7
\linethickness=0.7pt
\ninerm
  \Red{\relax
  \putrule from 0 0.0 to 0.25 0.0 }\relax
  \put {iteration 10} [l] at 0.35 0.0
  \putrule from 0 0.2 to 0.25 0.2  
  \OliveGreen{\relax
  \putrule from 0 0.2 to 0.25 0.2 }\relax
  \put {iteration 3} [l] at 0.35 0.2
  \Blue{\relax
   \putrule from 0 0.4 to 0.25 0.4 }\relax
  \putrule from 0 0.4 to 0.2 0.4
  \put {iteration 1} [l] at 0.35 0.4
\setsolid
  \putrule from 0 0.6 to 0.25 0.6
  \put {actual $q$} [l] at 0.35 0.6
\endpicture
\relax
}
\setbox\figureone=\vbox{\hsize=\xfiglen
\beginpicture
\eightrm
  \setcoordinatesystem units <0.033\xfiglen,0.15\yfiglen>  point at 0 -4
  \setplotarea x from 0 to 31, y from -4 to 3
  \axis bottom shiftedto y=-4 ticks short numbered from 5 to 30 by 5 /
  \axis left ticks short  from -4 to 3 by 1 /
  \put {{\eightrm 3}} [r] at -0.8 3
  \put {{\eightrm 2}} [r] at -0.8 2
  \put {{\eightrm 1}} [r] at -0.8 1
  \put {{\eightrm 0}} [r] at -0.8 0
  \put {{\eightrm -1}} [r] at -0.8 -1
  \put {{\eightrm -2}} [r] at -0.8 -2
  \put {{\eightrm -3}} [r] at -0.8 -3
  \put {{\eightrm -4}} [r] at -0.8 -4
\small
\eightrm
\put {\copy\figurelegendone} [lb] at 1 -3
\multiput {\Red{\relax$\star$}\relax} at
    1.0000    2.9111
    2.0000    2.9081
    3.0000    2.8013
    4.0000    2.7962
    5.0000    2.6789
    6.0000    2.6719
    7.0000    2.5422
    8.0000    2.5335
    9.0000    2.3885
   10.0000    2.3789
   11.0000    2.2147
   12.0000    2.2073
   13.0000    2.0218
   14.0000    2.0214
   15.0000    1.8577
   16.0000    1.7827
   17.0000    1.6265
   18.0000    1.4801
   19.0000    1.3035
   20.0000    1.1207
   21.0000    0.9204
   22.0000    0.7103
   23.0000    0.4869
   24.0000    0.2526
   25.0000    0.0065
   26.0000   -0.2506
   27.0000   -0.5191
   28.0000   -0.7990
   29.0000   -1.0903
   30.0000   -1.3933
   31.0000   -1.7084
 /
\multiput {\Blue{\relax$\circ$}\relax} at
    1.0000    1.7518
    2.0000    1.6811
    3.0000    1.6065
    4.0000    1.5281
    5.0000    1.4459
    6.0000    1.3593
    7.0000    1.2682
    8.0000    1.1719
    9.0000    1.0696
   10.0000    0.9601
   11.0000    0.8420
   12.0000    0.7142
   13.0000    0.5757
   14.0000    0.4263
   15.0000    0.2659
   16.0000    0.0947
   17.0000   -0.0874
   18.0000   -0.2800
   19.0000   -0.4832
   20.0000   -0.6967
   21.0000   -0.9206
   22.0000   -1.1547
   23.0000   -1.3990
   24.0000   -1.6535
   25.0000   -1.9183
   26.0000   -2.1934
   27.0000   -2.4789
   28.0000   -2.7750
   29.0000   -3.0819
   30.0000   -3.4000
   31.0000   -3.7295
 /
\endpicture
}
\setbox\figureoneb=\vbox{\hsize=\xfiglen
\beginpicture
\eightrm
  \setcoordinatesystem units <0.9\xfiglen,0.007\yfiglen>  point at 0 -60
  \setplotarea x from 0 to 1, y from -60 to 90
  \axis bottom shiftedto y=0 ticks short numbered from 1 to 1 by 0.2 /
  \axis left ticks short from -60 to 90 by 30 /
  \put {{\eightrm -60}} [r] at -0.02 -60
  \put {{\eightrm -30}} [r] at -0.02 -30
  \put {{\eightrm 0}} [r] at -0.02 0
  \put {{\eightrm 30}} [r] at -0.02 30
  \put {{\eightrm 60}} [r] at -0.02 60
  \put {{\eightrm 90}} [r] at -0.02 90
\setsolid
\small
\Blue{\relax \putrule from 0.06 82 to 0.15 82 \relax}\relax
\put {{\footnotesize $W$}} [l] at 0.17 82
\setdashes <3pt>
\Red{\relax \putrule from 0.06 67 to 0.15 67 \relax}\relax
\put {{\footnotesize $W'$}} [l] at 0.17 67
\setsolid
\Blue{\relax 
\plot
        0    1.5646
    0.0200    1.5547
    0.0400    1.5163
    0.0600    1.4404
    0.0800    1.3183
    0.1000    1.1419
    0.1200    0.9049
    0.1400    0.6021
    0.1600    0.2288
    0.1800   -0.2202
    0.2000   -0.7486
    0.2200   -1.3567
    0.2400   -2.0405
    0.2600   -2.7936
    0.2800   -3.6108
    0.3000   -4.4885
    0.3200   -5.4241
    0.3400   -6.4128
    0.3600   -7.4460
    0.3800   -8.5117
    0.4000   -9.5949
    0.4200  -10.6795
    0.4400  -11.7494
    0.4600  -12.7903
    0.4800  -13.7904
    0.5000  -14.7388
    0.5200  -15.6237
    0.5400  -16.4320
    0.5600  -17.1505
    0.5800  -17.7671
    0.6000  -18.2708
    0.6200  -18.6511
    0.6400  -18.8982
    0.6600  -19.0032
    0.6800  -18.9601
    0.7000  -18.7649
    0.7200  -18.4155
    0.7400  -17.9108
    0.7600  -17.2512
    0.7800  -16.4388
    0.8000  -15.4773
    0.8200  -14.3724
    0.8400  -13.1321
    0.8600  -11.7676
    0.8800  -10.2915
    0.9000   -8.7177
    0.9200   -7.0615
    0.9400   -5.3419
    0.9600   -3.5796
    0.9800   -1.7943
    1.0000   -0.0000
 /\relax}\relax
\setdashes <4pt>
\Red{\relax 
\plot
         0         0
    0.0200   -1.2284
    0.0400   -2.9053
    0.0600   -5.0088
    0.0800   -7.5092
    0.1000  -10.3688
    0.1200  -13.5414
    0.1400  -16.9723
    0.1600  -20.6003
    0.1800  -24.3593
    0.2000  -28.1791
    0.2200  -31.9874
    0.2400  -35.7113
    0.2600  -39.2787
    0.2800  -42.6162
    0.3000  -45.6476
    0.3200  -48.2965
    0.3400  -50.4915
    0.3600  -52.1711
    0.3800  -53.2840
    0.4000  -53.7866
    0.4200  -53.6431
    0.4400  -52.8274
    0.4600  -51.3237
    0.4800  -49.1236
    0.5000  -46.2256
    0.5200  -42.6385
    0.5400  -38.3856
    0.5600  -33.5035
    0.5800  -28.0358
    0.6000  -22.0294
    0.6200  -15.5390
    0.6400   -8.6312
    0.6600   -1.3838
    0.6800    6.1219
    0.7000   13.8048
    0.7200   21.5786
    0.7400   29.3478
    0.7600   37.0113
    0.7800   44.4728
    0.8000   51.6454
    0.8200   58.4457
    0.8400   64.7861
    0.8600   70.5768
    0.8800   75.7357
    0.9000   80.1965
    0.9200   83.9075
    0.9400   86.8273
    0.9600   88.9261
    0.9800   90.1878
    1.0000   90.6084
 /\relax}\relax
\endpicture
}
\setbox\figuretwo=\vbox{\hsize=\xfiglen
\beginpicture
\eightrm
  \setcoordinatesystem units <\xfiglen,0.75\yfiglen>  point at 0 0.5
  \setplotarea x from 0 to 1, y from 0.5 to 2.2
  \axis bottom shiftedto y=0.5 ticks short numbered from 0  to 1 by 0.2 /
  \axis left ticks short numbered from 0.5 to 2 by 0.5 /
\put {{\small $a(x)$}} [lt] at 0.02 2.2
\linethickness=0.8pt
\setsolid
\plot  
    0.00    1.0000
    0.01    1.0631
    0.02    1.1259
    0.03    1.1876
    0.04    1.2470
    0.05    1.3034
    0.06    1.3558
    0.07    1.4035
    0.08    1.4457
    0.09    1.4819
    0.10    1.5115
    0.11    1.5342
    0.12    1.5497
    0.13    1.5578
    0.14    1.5586
    0.15    1.5520
    0.16    1.5384
    0.17    1.5181
    0.18    1.4915
    0.19    1.4592
    0.20    1.4219
    0.21    1.3802
    0.22    1.3351
    0.23    1.2873
    0.24    1.2378
    0.25    1.1875
    0.26    1.1374
    0.27    1.0885
    0.28    1.0417
    0.29    0.9980
    0.30    0.9581
    0.31    0.9230
    0.32    0.8933
    0.33    0.8697
    0.34    0.8528
    0.35    0.8430
    0.36    0.8406
    0.37    0.8460
    0.38    0.8591
    0.39    0.8800
    0.40    0.9085
    0.41    0.9443
    0.42    0.9871
    0.43    1.0363
    0.44    1.0914
    0.45    1.1516
    0.46    1.2162
    0.47    1.2842
    0.48    1.3549
    0.49    1.4271
    0.50    1.5000
    0.51    1.5725
    0.52    1.6435
    0.53    1.7122
    0.54    1.7774
    0.55    1.8384
    0.56    1.8942
    0.57    1.9441
    0.58    1.9873
    0.59    2.0233
    0.60    2.0515
    0.61    2.0716
    0.62    2.0833
    0.63    2.0864
    0.64    2.0810
    0.65    2.0670
    0.66    2.0448
    0.67    2.0147
    0.68    1.9771
    0.69    1.9326
    0.70    1.8819
    0.71    1.8256
    0.72    1.7647
    0.73    1.6999
    0.74    1.6322
    0.75    1.5625
    0.76    1.4918
    0.77    1.4211
    0.78    1.3513
    0.79    1.2834
    0.80    1.2181
    0.81    1.1564
    0.82    1.0989
    0.83    1.0463
    0.84    0.9992
    0.85    0.9580
    0.86    0.9230
    0.87    0.8946
    0.88    0.8727
    0.89    0.8574
    0.90    0.8485
    0.91    0.8457
    0.92    0.8487
    0.93    0.8569
    0.94    0.8698
    0.95    0.8866
    0.96    0.9066
    0.97    0.9288
    0.98    0.9525
    0.99    0.9765
    1.00    1.0000
 /
\setlinear
\setdots <2pt>
\Red{\relax
\plot
         0    1.0073
    0.0100    1.0705
    0.0200    1.1335
    0.0300    1.1952
    0.0400    1.2548
    0.0500    1.3111
    0.0600    1.3634
    0.0700    1.4110
    0.0800    1.4531
    0.0900    1.4890
    0.1000    1.5183
    0.1100    1.5407
    0.1200    1.5558
    0.1300    1.5635
    0.1400    1.5639
    0.1500    1.5569
    0.1600    1.5428
    0.1700    1.5220
    0.1800    1.4950
    0.1900    1.4622
    0.2000    1.4244
    0.2100    1.3823
    0.2200    1.3366
    0.2300    1.2884
    0.2400    1.2384
    0.2500    1.1877
    0.2600    1.1371
    0.2700    1.0877
    0.2800    1.0404
    0.2900    0.9962
    0.3000    0.9559
    0.3100    0.9204
    0.3200    0.8903
    0.3300    0.8664
    0.3400    0.8492
    0.3500    0.8391
    0.3600    0.8366
    0.3700    0.8417
    0.3800    0.8546
    0.3900    0.8753
    0.4000    0.9036
    0.4100    0.9392
    0.4200    0.9817
    0.4300    1.0307
    0.4400    1.0855
    0.4500    1.1455
    0.4600    1.2098
    0.4700    1.2777
    0.4800    1.3482
    0.4900    1.4204
    0.5000    1.4932
    0.5100    1.5657
    0.5200    1.6369
    0.5300    1.7057
    0.5400    1.7711
    0.5500    1.8322
    0.5600    1.8883
    0.5700    1.9384
    0.5800    1.9819
    0.5900    2.0181
    0.6000    2.0466
    0.6100    2.0670
    0.6200    2.0790
    0.6300    2.0825
    0.6400    2.0773
    0.6500    2.0637
    0.6600    2.0419
    0.6700    2.0121
    0.6800    1.9750
    0.6900    1.9309
    0.7000    1.8805
    0.7100    1.8247
    0.7200    1.7642
    0.7300    1.6998
    0.7400    1.6326
    0.7500    1.5634
    0.7600    1.4932
    0.7700    1.4230
    0.7800    1.3538
    0.7900    1.2863
    0.8000    1.2216
    0.8100    1.1604
    0.8200    1.1034
    0.8300    1.0513
    0.8400    1.0047
    0.8500    0.9640
    0.8600    0.9294
    0.8700    0.9014
    0.8800    0.8799
    0.8900    0.8649
    0.9000    0.8563
    0.9100    0.8538
    0.9200    0.8570
    0.9300    0.8655
    0.9400    0.8786
    0.9500    0.8957
    0.9600    0.9158
    0.9700    0.9383
    0.9800    0.9621
    0.9900    0.9861
    1.0000    1.0083
 /\relax}\relax
\setdashes <5pt>
\Blue{\relax 
\plot
        0    1.0065
    0.0100    1.0697
    0.0200    1.1328
    0.0300    1.1946
    0.0400    1.2542
    0.0500    1.3107
    0.0600    1.3632
    0.0700    1.4109
    0.0800    1.4531
    0.0900    1.4892
    0.1000    1.5187
    0.1100    1.5412
    0.1200    1.5565
    0.1300    1.5644
    0.1400    1.5649
    0.1500    1.5581
    0.1600    1.5443
    0.1700    1.5237
    0.1800    1.4968
    0.1900    1.4643
    0.2000    1.4266
    0.2100    1.3847
    0.2200    1.3393
    0.2300    1.2912
    0.2400    1.2415
    0.2500    1.1909
    0.2600    1.1406
    0.2700    1.0914
    0.2800    1.0443
    0.2900    1.0003
    0.3000    0.9602
    0.3100    0.9249
    0.3200    0.8950
    0.3300    0.8713
    0.3400    0.8542
    0.3500    0.8444
    0.3600    0.8420
    0.3700    0.8473
    0.3800    0.8604
    0.3900    0.8813
    0.4000    0.9098
    0.4100    0.9456
    0.4200    0.9884
    0.4300    1.0376
    0.4400    1.0926
    0.4500    1.1528
    0.4600    1.2174
    0.4700    1.2855
    0.4800    1.3563
    0.4900    1.4286
    0.5000    1.5017
    0.5100    1.5743
    0.5200    1.6456
    0.5300    1.7145
    0.5400    1.7800
    0.5500    1.8412
    0.5600    1.8972
    0.5700    1.9473
    0.5800    1.9908
    0.5900    2.0269
    0.6000    2.0553
    0.6100    2.0755
    0.6200    2.0873
    0.6300    2.0905
    0.6400    2.0850
    0.6500    2.0711
    0.6600    2.0489
    0.6700    2.0187
    0.6800    1.9811
    0.6900    1.9365
    0.7000    1.8857
    0.7100    1.8293
    0.7200    1.7683
    0.7300    1.7034
    0.7400    1.6356
    0.7500    1.5658
    0.7600    1.4950
    0.7700    1.4242
    0.7800    1.3544
    0.7900    1.2864
    0.8000    1.2211
    0.8100    1.1593
    0.8200    1.1018
    0.8300    1.0492
    0.8400    1.0021
    0.8500    0.9609
    0.8600    0.9259
    0.8700    0.8975
    0.8800    0.8756
    0.8900    0.8603
    0.9000    0.8514
    0.9100    0.8486
    0.9200    0.8516
    0.9300    0.8599
    0.9400    0.8728
    0.9500    0.8898
    0.9600    0.9098
    0.9700    0.9323
    0.9800    0.9562
    0.9900    0.9805
    1.0000    1.0035
 /\relax}\relax
\endpicture}
\setbox\figurethree=\vbox{\hsize=\xfiglen
%
\beginpicture
\eightrm
  \setcoordinatesystem units <\xfiglen,0.8\yfiglen>  point at 0 -0.25
  \setplotarea x from 0 to 1, y from -0.25 to 1
  \axis bottom shiftedto y=-0.25 ticks short numbered from 0 to 1 by 0.2 /
  \axis left ticks short numbered from 0 to 1 by 0.5 /
\put {{\small $q(x)$}} [lt] at 0.02 1
\setsolid 
\plot
         0         0
    0.0100    0.0776
    0.0200    0.1507
    0.0300    0.2193
    0.0400    0.2838
    0.0500    0.3443
    0.0600    0.4009
    0.0700    0.4539
    0.0800    0.5034
    0.0900    0.5496
    0.1000    0.5927
    0.1100    0.6327
    0.1200    0.6698
    0.1300    0.7041
    0.1400    0.7359
    0.1500    0.7652
    0.1600    0.7920
    0.1700    0.8167
    0.1800    0.8392
    0.1900    0.8596
    0.2000    0.8781
    0.2100    0.8948
    0.2200    0.9097
    0.2300    0.9229
    0.2400    0.9346
    0.2500    0.9447
    0.2600    0.9535
    0.2700    0.9609
    0.2800    0.9670
    0.2900    0.9720
    0.3000    0.9758
    0.3100    0.9785
    0.3200    0.9802
    0.3300    0.9810
    0.3400    0.9808
    0.3500    0.9798
    0.3600    0.9780
    0.3700    0.9755
    0.3800    0.9722
    0.3900    0.9683
    0.4000    0.9638
    0.4100    0.9587
    0.4200    0.9531
    0.4300    0.9469
    0.4400    0.9403
    0.4500    0.9333
    0.4600    0.9258
    0.4700    0.9180
    0.4800    0.9098
    0.4900    0.9013
    0.5000    0.8925
    0.5100    0.8835
    0.5200    0.8742
    0.5300    0.8646
    0.5400    0.8549
    0.5500    0.8450
    0.5600    0.8350
    0.5700    0.8247
    0.5800    0.8144
    0.5900    0.8040
    0.6000    0.7934
    0.6100    0.7828
    0.6200    0.7721
    0.6300    0.7614
    0.6400    0.7506
    0.6500    0.7398
    0.6600    0.7290
    0.6700    0.7182
    0.6800    0.7074
    0.6900    0.6965
    0.7000    0.6858
    0.7100    0.6750
    0.7200    0.6643
    0.7300    0.6536
    0.7400    0.6430
    0.7500    0.6324
    0.7600    0.6219
    0.7700    0.6114
    0.7800    0.6011
    0.7900    0.5908
    0.8000    0.5806
    0.8100    0.5705
    0.8200    0.5605
    0.8300    0.5505
    0.8400    0.5407
    0.8500    0.5310
    0.8600    0.5213
    0.8700    0.5118
    0.8800    0.5024
    0.8900    0.4931
    0.9000    0.4839
    0.9100    0.4748
    0.9200    0.4658
    0.9300    0.4570
    0.9400    0.4482
    0.9500    0.4396
    0.9600    0.4311
    0.9700    0.4227
    0.9800    0.4145
    0.9900    0.4063
    1.0000    0.3983
/
\setdots <2pt>
\Red{\relax 
\plot
    0.00     -0.1668
    0.0100   -0.1420
    0.0200   -0.1173
    0.0300   -0.0926
    0.0400   -0.0679
    0.0500   -0.0431
    0.0600   -0.0184
    0.0700    0.0064
    0.0800    0.0313
    0.0900    0.0562
    0.1000    0.0812
    0.1100    0.1062
    0.1200    0.1312
    0.1300    0.1563
    0.1400    0.1813
    0.1500    0.2064
    0.1600    0.2313
    0.1700    0.2562
    0.1800    0.2809
    0.1900    0.3053
    0.2000    0.3295
    0.2100    0.3533
    0.2200    0.3766
    0.2300    0.3994
    0.2400    0.4214
    0.2500    0.4426
    0.2600    0.4629
    0.2700    0.4821
    0.2800    0.5000
    0.2900    0.5165
    0.3000    0.5314
    0.3100    0.5447
    0.3200    0.5562
    0.3300    0.5658
    0.3400    0.5733
    0.3500    0.5789
    0.3600    0.5824
    0.3700    0.5838
    0.3800    0.5832
    0.3900    0.5808
    0.4000    0.5766
    0.4100    0.5709
    0.4200    0.5639
    0.4300    0.5557
    0.4400    0.5465
    0.4500    0.5365
    0.4600    0.5258
    0.4700    0.5145
    0.4800    0.5024
    0.4900    0.4896
    0.5000    0.4760
    0.5100    0.4616
    0.5200    0.4463
    0.5300    0.4300
    0.5400    0.4127
    0.5500    0.3945
    0.5600    0.3754
    0.5700    0.3555
    0.5800    0.3349
    0.5900    0.3137
    0.6000    0.2921
    0.6100    0.2702
    0.6200    0.2483
    0.6300    0.2264
    0.6400    0.2047
    0.6500    0.1833
    0.6600    0.1623
    0.6700    0.1418
    0.6800    0.1217
    0.6900    0.1021
    0.7000    0.0829
    0.7100    0.0642
    0.7200    0.0458
    0.7300    0.0279
    0.7400    0.0103
    0.7500   -0.0071
    0.7600   -0.0242
    0.7700   -0.0414
    0.7800   -0.0585
    0.7900   -0.0757
    0.8000   -0.0929
    0.8100   -0.1103
    0.8200   -0.1279
    0.8300   -0.1454
    0.8400   -0.1628
    0.8500   -0.1798
    0.8600   -0.1963
    0.8700   -0.2124
    0.8800   -0.2278
    0.8900   -0.2427
    0.9000   -0.2567
    0.9100   -0.2699
    0.9200   -0.2823
    0.9300   -0.2939
    0.9400   -0.3050
    0.9500   -0.3156
    0.9600   -0.3257
    0.9700   -0.3349
    0.9800   -0.3425
    0.9900   -0.3477
    1.0000   -0.3492
/\relax}\relax
%
\setdashes <4pt>
\Blue{\relax 
\plot
         0    0.3549
    0.0100    0.3771
    0.0200    0.3993
    0.0300    0.4215
    0.0400    0.4437
    0.0500    0.4658
    0.0600    0.4879
    0.0700    0.5100
    0.0800    0.5319
    0.0900    0.5538
    0.1000    0.5755
    0.1100    0.5970
    0.1200    0.6184
    0.1300    0.6396
    0.1400    0.6605
    0.1500    0.6811
    0.1600    0.7014
    0.1700    0.7213
    0.1800    0.7409
    0.1900    0.7599
    0.2000    0.7785
    0.2100    0.7965
    0.2200    0.8140
    0.2300    0.8307
    0.2400    0.8467
    0.2500    0.8619
    0.2600    0.8762
    0.2700    0.8896
    0.2800    0.9019
    0.2900    0.9131
    0.3000    0.9231
    0.3100    0.9318
    0.3200    0.9392
    0.3300    0.9452
    0.3400    0.9498
    0.3500    0.9530
    0.3600    0.9547
    0.3700    0.9550
    0.3800    0.9539
    0.3900    0.9515
    0.4000    0.9481
    0.4100    0.9438
    0.4200    0.9387
    0.4300    0.9331
    0.4400    0.9272
    0.4500    0.9210
    0.4600    0.9146
    0.4700    0.9082
    0.4800    0.9016
    0.4900    0.8948
    0.5000    0.8878
    0.5100    0.8805
    0.5200    0.8728
    0.5300    0.8646
    0.5400    0.8559
    0.5500    0.8466
    0.5600    0.8369
    0.5700    0.8267
    0.5800    0.8161
    0.5900    0.8051
    0.6000    0.7940
    0.6100    0.7828
    0.6200    0.7717
    0.6300    0.7607
    0.6400    0.7499
    0.6500    0.7393
    0.6600    0.7291
    0.6700    0.7191
    0.6800    0.7095
    0.6900    0.7002
    0.7000    0.6910
    0.7100    0.6819
    0.7200    0.6730
    0.7300    0.6641
    0.7400    0.6552
    0.7500    0.6462
    0.7600    0.6370
    0.7700    0.6276
    0.7800    0.6178
    0.7900    0.6076
    0.8000    0.5968
    0.8100    0.5855
    0.8200    0.5737
    0.8300    0.5616
    0.8400    0.5493
    0.8500    0.5369
    0.8600    0.5247
    0.8700    0.5126
    0.8800    0.5008
    0.8900    0.4895
    0.9000    0.4787
    0.9100    0.4684
    0.9200    0.4586
    0.9300    0.4491
    0.9400    0.4400
    0.9500    0.4313
    0.9600    0.4230
    0.9700    0.4149
    0.9800    0.4068
    0.9900    0.3986
    1.0000    0.3908
/\relax}\relax
\endpicture
}
\setbox\figurefour=\vbox{\hsize=\xfiglen
\beginpicture
\eightrm
  \setcoordinatesystem units <\xfiglen,0.7\yfiglen>  point at 0 0.5
  \setplotarea x from 0 to 1, y from 0.5 to 2.2
  \axis bottom shiftedto y=0.5 ticks short numbered from 0  to 1 by 0.2 /
  \axis left ticks short numbered from 0.5 to 2 by 0.5 /
\put {{\small $a(x)$}} [lt] at 0.02 2.2
\linethickness=0.8pt
\setsolid
\plot  
    0.00    1.0000
    0.01    1.0631
    0.02    1.1259
    0.03    1.1876
    0.04    1.2470
    0.05    1.3034
    0.06    1.3558
    0.07    1.4035
    0.08    1.4457
    0.09    1.4819
    0.10    1.5115
    0.11    1.5342
    0.12    1.5497
    0.13    1.5578
    0.14    1.5586
    0.15    1.5520
    0.16    1.5384
    0.17    1.5181
    0.18    1.4915
    0.19    1.4592
    0.20    1.4219
    0.21    1.3802
    0.22    1.3351
    0.23    1.2873
    0.24    1.2378
    0.25    1.1875
    0.26    1.1374
    0.27    1.0885
    0.28    1.0417
    0.29    0.9980
    0.30    0.9581
    0.31    0.9230
    0.32    0.8933
    0.33    0.8697
    0.34    0.8528
    0.35    0.8430
    0.36    0.8406
    0.37    0.8460
    0.38    0.8591
    0.39    0.8800
    0.40    0.9085
    0.41    0.9443
    0.42    0.9871
    0.43    1.0363
    0.44    1.0914
    0.45    1.1516
    0.46    1.2162
    0.47    1.2842
    0.48    1.3549
    0.49    1.4271
    0.50    1.5000
    0.51    1.5725
    0.52    1.6435
    0.53    1.7122
    0.54    1.7774
    0.55    1.8384
    0.56    1.8942
    0.57    1.9441
    0.58    1.9873
    0.59    2.0233
    0.60    2.0515
    0.61    2.0716
    0.62    2.0833
    0.63    2.0864
    0.64    2.0810
    0.65    2.0670
    0.66    2.0448
    0.67    2.0147
    0.68    1.9771
    0.69    1.9326
    0.70    1.8819
    0.71    1.8256
    0.72    1.7647
    0.73    1.6999
    0.74    1.6322
    0.75    1.5625
    0.76    1.4918
    0.77    1.4211
    0.78    1.3513
    0.79    1.2834
    0.80    1.2181
    0.81    1.1564
    0.82    1.0989
    0.83    1.0463
    0.84    0.9992
    0.85    0.9580
    0.86    0.9230
    0.87    0.8946
    0.88    0.8727
    0.89    0.8574
    0.90    0.8485
    0.91    0.8457
    0.92    0.8487
    0.93    0.8569
    0.94    0.8698
    0.95    0.8866
    0.96    0.9066
    0.97    0.9288
    0.98    0.9525
    0.99    0.9765
    1.00    1.0000
 /
\setlinear
\setdots <2pt>
\Red{\relax
\plot
         0    1.0311
    0.0100    1.0939
    0.0200    1.1565
    0.0300    1.2178
    0.0400    1.2767
    0.0500    1.3325
    0.0600    1.3842
    0.0700    1.4311
    0.0800    1.4726
    0.0900    1.5082
    0.1000    1.5374
    0.1100    1.5596
    0.1200    1.5746
    0.1300    1.5821
    0.1400    1.5817
    0.1500    1.5736
    0.1600    1.5578
    0.1700    1.5348
    0.1800    1.5051
    0.1900    1.4696
    0.2000    1.4293
    0.2100    1.3852
    0.2200    1.3382
    0.2300    1.2893
    0.2400    1.2396
    0.2500    1.1897
    0.2600    1.1406
    0.2700    1.0929
    0.2800    1.0474
    0.2900    1.0049
    0.3000    0.9660
    0.3100    0.9315
    0.3200    0.9023
    0.3300    0.8788
    0.3400    0.8618
    0.3500    0.8517
    0.3600    0.8490
    0.3700    0.8538
    0.3800    0.8664
    0.3900    0.8867
    0.4000    0.9147
    0.4100    0.9501
    0.4200    0.9924
    0.4300    1.0414
    0.4400    1.0963
    0.4500    1.1565
    0.4600    1.2211
    0.4700    1.2894
    0.4800    1.3603
    0.4900    1.4329
    0.5000    1.5061
    0.5100    1.5789
    0.5200    1.6503
    0.5300    1.7193
    0.5400    1.7849
    0.5500    1.8462
    0.5600    1.9023
    0.5700    1.9525
    0.5800    1.9960
    0.5900    2.0323
    0.6000    2.0608
    0.6100    2.0811
    0.6200    2.0930
    0.6300    2.0963
    0.6400    2.0911
    0.6500    2.0775
    0.6600    2.0556
    0.6700    2.0259
    0.6800    1.9888
    0.6900    1.9449
    0.7000    1.8947
    0.7100    1.8390
    0.7200    1.7787
    0.7300    1.7146
    0.7400    1.6477
    0.7500    1.5789
    0.7600    1.5092
    0.7700    1.4395
    0.7800    1.3710
    0.7900    1.3044
    0.8000    1.2407
    0.8100    1.1807
    0.8200    1.1252
    0.8300    1.0749
    0.8400    1.0302
    0.8500    0.9918
    0.8600    0.9601
    0.8700    0.9352
    0.8800    0.9175
    0.8900    0.9071
    0.9000    0.9040
    0.9100    0.9084
    0.9200    0.9200
    0.9300    0.9388
    0.9400    0.9647
    0.9500    0.9977
    0.9600    1.0377
    0.9700    1.0847
    0.9800    1.1384
    0.9900    1.1990
    1.0000    1.2593
 /\relax}\relax
\setdashes <5pt>
\Blue{\relax 
\plot
         0    1.0230
    0.0100    1.0859
    0.0200    1.1484
    0.0300    1.2096
    0.0400    1.2686
    0.0500    1.3242
    0.0600    1.3758
    0.0700    1.4227
    0.0800    1.4641
    0.0900    1.4996
    0.1000    1.5286
    0.1100    1.5507
    0.1200    1.5656
    0.1300    1.5729
    0.1400    1.5724
    0.1500    1.5641
    0.1600    1.5481
    0.1700    1.5249
    0.1800    1.4952
    0.1900    1.4596
    0.2000    1.4191
    0.2100    1.3748
    0.2200    1.3277
    0.2300    1.2788
    0.2400    1.2289
    0.2500    1.1790
    0.2600    1.1298
    0.2700    1.0820
    0.2800    1.0365
    0.2900    0.9939
    0.3000    0.9551
    0.3100    0.9207
    0.3200    0.8915
    0.3300    0.8682
    0.3400    0.8513
    0.3500    0.8413
    0.3600    0.8388
    0.3700    0.8438
    0.3800    0.8566
    0.3900    0.8772
    0.4000    0.9054
    0.4100    0.9410
    0.4200    0.9837
    0.4300    1.0329
    0.4400    1.0881
    0.4500    1.1486
    0.4600    1.2135
    0.4700    1.2821
    0.4800    1.3533
    0.4900    1.4261
    0.5000    1.4995
    0.5100    1.5726
    0.5200    1.6442
    0.5300    1.7133
    0.5400    1.7790
    0.5500    1.8404
    0.5600    1.8965
    0.5700    1.9467
    0.5800    1.9902
    0.5900    2.0263
    0.6000    2.0546
    0.6100    2.0747
    0.6200    2.0863
    0.6300    2.0892
    0.6400    2.0836
    0.6500    2.0695
    0.6600    2.0472
    0.6700    2.0169
    0.6800    1.9791
    0.6900    1.9344
    0.7000    1.8834
    0.7100    1.8269
    0.7200    1.7657
    0.7300    1.7007
    0.7400    1.6327
    0.7500    1.5628
    0.7600    1.4919
    0.7700    1.4210
    0.7800    1.3510
    0.7900    1.2830
    0.8000    1.2178
    0.8100    1.1562
    0.8200    1.0989
    0.8300    1.0465
    0.8400    0.9998
    0.8500    0.9590
    0.8600    0.9245
    0.8700    0.8967
    0.8800    0.8755
    0.8900    0.8610
    0.9000    0.8531
    0.9100    0.8516
    0.9200    0.8564
    0.9300    0.8674
    0.9400    0.8846
    0.9500    0.9078
    0.9600    0.9370
    0.9700    0.9720
    0.9800    1.0129
    0.9900    1.0596
    1.0000    1.1065
 /\relax}\relax
\endpicture}
\setbox\figurefive=\vbox{\hsize=\xfiglen
%
\beginpicture
\eightrm
  \setcoordinatesystem units <\xfiglen,0.40\yfiglen>  point at 0 -0.25
  \setplotarea x from 0 to 1, y from -0.25 to 2.7
  \axis bottom shiftedto y=-0.25 ticks short numbered from 0 to 1 by 0.2 /
  \axis left ticks short numbered from 0 to 2.5 by 0.5 /
\put {{\small $q(x)$}} [lt] at 0.02 2.7
\setsolid 
\plot
         0         0
    0.0100    0.0776
    0.0200    0.1507
    0.0300    0.2193
    0.0400    0.2838
    0.0500    0.3443
    0.0600    0.4009
    0.0700    0.4539
    0.0800    0.5034
    0.0900    0.5496
    0.1000    0.5927
    0.1100    0.6327
    0.1200    0.6698
    0.1300    0.7041
    0.1400    0.7359
    0.1500    0.7652
    0.1600    0.7920
    0.1700    0.8167
    0.1800    0.8392
    0.1900    0.8596
    0.2000    0.8781
    0.2100    0.8948
    0.2200    0.9097
    0.2300    0.9229
    0.2400    0.9346
    0.2500    0.9447
    0.2600    0.9535
    0.2700    0.9609
    0.2800    0.9670
    0.2900    0.9720
    0.3000    0.9758
    0.3100    0.9785
    0.3200    0.9802
    0.3300    0.9810
    0.3400    0.9808
    0.3500    0.9798
    0.3600    0.9780
    0.3700    0.9755
    0.3800    0.9722
    0.3900    0.9683
    0.4000    0.9638
    0.4100    0.9587
    0.4200    0.9531
    0.4300    0.9469
    0.4400    0.9403
    0.4500    0.9333
    0.4600    0.9258
    0.4700    0.9180
    0.4800    0.9098
    0.4900    0.9013
    0.5000    0.8925
    0.5100    0.8835
    0.5200    0.8742
    0.5300    0.8646
    0.5400    0.8549
    0.5500    0.8450
    0.5600    0.8350
    0.5700    0.8247
    0.5800    0.8144
    0.5900    0.8040
    0.6000    0.7934
    0.6100    0.7828
    0.6200    0.7721
    0.6300    0.7614
    0.6400    0.7506
    0.6500    0.7398
    0.6600    0.7290
    0.6700    0.7182
    0.6800    0.7074
    0.6900    0.6965
    0.7000    0.6858
    0.7100    0.6750
    0.7200    0.6643
    0.7300    0.6536
    0.7400    0.6430
    0.7500    0.6324
    0.7600    0.6219
    0.7700    0.6114
    0.7800    0.6011
    0.7900    0.5908
    0.8000    0.5806
    0.8100    0.5705
    0.8200    0.5605
    0.8300    0.5505
    0.8400    0.5407
    0.8500    0.5310
    0.8600    0.5213
    0.8700    0.5118
    0.8800    0.5024
    0.8900    0.4931
    0.9000    0.4839
    0.9100    0.4748
    0.9200    0.4658
    0.9300    0.4570
    0.9400    0.4482
    0.9500    0.4396
    0.9600    0.4311
    0.9700    0.4227
    0.9800    0.4145
    0.9900    0.4063
    1.0000    0.3983
/
\setdots <2pt>
\Red{\relax 
\plot
         0   -0.4941
    0.0100   -0.4621
    0.0200   -0.4751
    0.0300   -0.5157
    0.0400   -0.5557
    0.0500   -0.5686
    0.0600   -0.5371
    0.0700   -0.4581
    0.0800   -0.2995
    0.0900   -0.0370
    0.1000    0.2905
    0.1100    0.6488
    0.1200    1.0086
    0.1300    1.3452
    0.1400    1.6390
    0.1500    1.8747
    0.1600    2.0421
    0.1700    2.1354
    0.1800    2.1533
    0.1900    2.0991
    0.2000    1.9803
    0.2100    1.8083
    0.2200    1.5984
    0.2300    1.3693
    0.2400    1.1429
    0.2500    0.9439
    0.2600    0.7993
    0.2700    0.7382
    0.2800    0.7248
    0.2900    0.6753
    0.3000    0.6059
    0.3100    0.5316
    0.3200    0.4633
    0.3300    0.4071
    0.3400    0.3651
    0.3500    0.3365
    0.3600    0.3193
    0.3700    0.3109
    0.3800    0.3095
    0.3900    0.3133
    0.4000    0.3215
    0.4100    0.3332
    0.4200    0.3475
    0.4300    0.3630
    0.4400    0.3780
    0.4500    0.3906
    0.4600    0.3989
    0.4700    0.4017
    0.4800    0.3983
    0.4900    0.3893
    0.5000    0.3757
    0.5100    0.3591
    0.5200    0.3409
    0.5300    0.3221
    0.5400    0.3032
    0.5500    0.2840
    0.5600    0.2642
    0.5700    0.2434
    0.5800    0.2219
    0.5900    0.2001
    0.6000    0.1790
    0.6100    0.1595
    0.6200    0.1422
    0.6300    0.1273
    0.6400    0.1145
    0.6500    0.1032
    0.6600    0.0925
    0.6700    0.0817
    0.6800    0.0706
    0.6900    0.0591
    0.7000    0.0474
    0.7100    0.0360
    0.7200    0.0252
    0.7300    0.0154
    0.7400    0.0067
    0.7500   -0.0011
    0.7600   -0.0082
    0.7700   -0.0149
    0.7800   -0.0215
    0.7900   -0.0284
    0.8000   -0.0358
    0.8100   -0.0440
    0.8200   -0.0529
    0.8300   -0.0626
    0.8400   -0.0730
    0.8500   -0.0836
    0.8600   -0.0943
    0.8700   -0.1047
    0.8800   -0.1146
    0.8900   -0.1239
    0.9000   -0.1327
    0.9100   -0.1431
    0.9200   -0.1563
    0.9300   -0.1736
    0.9400   -0.1978
    0.9500   -0.2331
    0.9600   -0.2849
    0.9700   -0.3595
    0.9800   -0.4640
    0.9900   -0.6063
    1.0000   -0.7945
/\relax}\relax
%
\setdashes <4pt>
\Blue{\relax 
\plot
         0   -0.4886
    0.0100   -0.4143
    0.0200   -0.3841
    0.0300   -0.3823
    0.0400   -0.3810
    0.0500   -0.3539
    0.0600   -0.2840
    0.0700   -0.1680
    0.0800    0.0261
    0.0900    0.3229
    0.1000    0.6833
    0.1100    1.0729
    0.1200    1.4624
    0.1300    1.8270
    0.1400    2.1467
    0.1500    2.4065
    0.1600    2.5958
    0.1700    2.7089
    0.1800    2.7445
    0.1900    2.7059
    0.2000    2.6006
    0.2100    2.4401
    0.2200    2.2398
    0.2300    2.0186
    0.2400    1.7985
    0.2500    1.6043
    0.2600    1.4632
    0.2700    1.4045
    0.2800    1.3927
    0.2900    1.3438
    0.3000    1.2741
    0.3100    1.1988
    0.3200    1.1285
    0.3300    1.0696
    0.3400    1.0242
    0.3500    0.9917
    0.3600    0.9701
    0.3700    0.9571
    0.3800    0.9508
    0.3900    0.9498
    0.4000    0.9532
    0.4100    0.9602
    0.4200    0.9700
    0.4300    0.9813
    0.4400    0.9926
    0.4500    1.0017
    0.4600    1.0071
    0.4700    1.0073
    0.4800    1.0019
    0.4900    0.9913
    0.5000    0.9766
    0.5100    0.9593
    0.5200    0.9409
    0.5300    0.9225
    0.5400    0.9043
    0.5500    0.8862
    0.5600    0.8679
    0.5700    0.8490
    0.5800    0.8296
    0.5900    0.8103
    0.6000    0.7918
    0.6100    0.7751
    0.6200    0.7608
    0.6300    0.7491
    0.6400    0.7395
    0.6500    0.7314
    0.6600    0.7240
    0.6700    0.7165
    0.6800    0.7086
    0.6900    0.7001
    0.7000    0.6914
    0.7100    0.6827
    0.7200    0.6746
    0.7300    0.6673
    0.7400    0.6608
    0.7500    0.6552
    0.7600    0.6501
    0.7700    0.6452
    0.7800    0.6402
    0.7900    0.6347
    0.8000    0.6284
    0.8100    0.6212
    0.8200    0.6130
    0.8300    0.6038
    0.8400    0.5939
    0.8500    0.5835
    0.8600    0.5729
    0.8700    0.5625
    0.8800    0.5525
    0.8900    0.5430
    0.9000    0.5339
    0.9100    0.5270
    0.9200    0.5240
    0.9300    0.5235
    0.9400    0.5227
    0.9500    0.5173
    0.9600    0.5021
    0.9700    0.4708
    0.9800    0.4164
    0.9900    0.3311
    1.0000    0.2066
/\relax}\relax
\endpicture
}
\setbox\figuresix=\vbox{\hsize=\xfiglen
\beginpicture
\eightrm
  \setcoordinatesystem units <\xfiglen,0.6\yfiglen>  point at 0 -1
  \setplotarea x from 0 to 1, y from -1.0 to 1.0
  \axis bottom shiftedto y=0 ticks short numbered from 0.2 to 1 by 0.2 /
  \axis left ticks short from -1 to 1 by 0.5 /
  \put {{\eightrm 1}} [r] at -0.02 1
  \put {{\eightrm 0}} [r] at -0.02 0
  \put {{\eightrm -1}} [r] at -0.02 -1
\setsolid
\small
\putrule from 0.1 -0.6 to 0.19 -0.6 \put {$u_0$} [l] at 0.22 -0.6
\plot       
         0         0
    0.0200    0.0314
    0.0400    0.0628
    0.0600    0.0941
    0.0800    0.1253
    0.1000    0.1564
    0.1200    0.1874
    0.1400    0.2181
    0.1600    0.2487
    0.1800    0.2790
    0.2000    0.3090
    0.2200    0.3387
    0.2400    0.3681
    0.2600    0.3971
    0.2800    0.4258
    0.3000    0.4540
    0.3200    0.4818
    0.3400    0.5090
    0.3600    0.5358
    0.3800    0.5621
    0.4000    0.5878
    0.4200    0.6129
    0.4400    0.6374
    0.4600    0.6613
    0.4800    0.6845
    0.5000    0.7071
    0.5200    0.7290
    0.5400    0.7501
    0.5600    0.7705
    0.5800    0.7902
    0.6000    0.8090
    0.6200    0.8271
    0.6400    0.8443
    0.6600    0.8607
    0.6800    0.8763
    0.7000    0.8910
    0.7200    0.9048
    0.7400    0.9178
    0.7600    0.9298
    0.7800    0.9409
    0.8000    0.9511
    0.8200    0.9603
    0.8400    0.9686
    0.8600    0.9759
    0.8800    0.9823
    0.9000    0.9877
    0.9200    0.9921
    0.9400    0.9956
    0.9600    0.9980
    0.9800    0.9995
    1.0000    1.0000
/
\setdashes <3pt>
\putrule from 0.1 -0.8 to 0.19 -0.8 \put {$v_0$} [l] at 0.22 -0.8
\plot       
    0.0200    0.0941
    0.0400    0.1874
    0.0600    0.2790
    0.0800    0.3681
    0.1000    0.4540
    0.1200    0.5358
    0.1400    0.6129
    0.1600    0.6845
    0.1800    0.7501
    0.2000    0.8090
    0.2200    0.8607
    0.2400    0.9048
    0.2600    0.9409
    0.2800    0.9686
    0.3000    0.9877
    0.3200    0.9980
    0.3400    0.9995
    0.3600    0.9921
    0.3800    0.9759
    0.4000    0.9511
    0.4200    0.9178
    0.4400    0.8763
    0.4600    0.8271
    0.4800    0.7705
    0.5000    0.7071
    0.5200    0.6374
    0.5400    0.5621
    0.5600    0.4818
    0.5800    0.3971
    0.6000    0.3090
    0.6200    0.2181
    0.6400    0.1253
    0.6600    0.0314
    0.6800   -0.0628
    0.7000   -0.1564
    0.7200   -0.2487
    0.7400   -0.3387
    0.7600   -0.4258
    0.7800   -0.5090
    0.8000   -0.5878
    0.8200   -0.6613
    0.8400   -0.7290
    0.8600   -0.7902
    0.8800   -0.8443
    0.9000   -0.8910
    0.9200   -0.9298
    0.9400   -0.9603
    0.9600   -0.9823
    0.9800   -0.9956
    1.0000   -1.0000
 /
\endpicture
}
\setbox\figureseven=\vbox{\hsize=\xfiglen
\beginpicture
\eightrm
  \setcoordinatesystem units <\xfiglen,0.27\yfiglen>  point at 0 -2
  \setplotarea x from 0 to 1, y from -2 to 2.5
  \axis bottom shiftedto y=0 ticks short numbered from 0.2 to 1 by 0.2 /
  \axis left ticks short from -2 to 2.5 by 0.5 /
  \put {{\eightrm 2}} [r] at -0.02 2
  \put {{\eightrm 1}} [r] at -0.02 1
  \put {{\eightrm 0}} [r] at -0.02 0
  \put {{\eightrm -1}} [r] at -0.02 -1
  \put {{\eightrm -2}} [r] at -0.02 -2
\setsolid
\Blue{\relax \putrule from 0.06 -1.4 to 0.15 -1.4 \relax}\relax
\put {$g_u$} [l] at 0.17 -1.4
\Blue{\relax 
\plot
         0    0.2500
    0.0200    0.3127
    0.0400    0.3753
    0.0600    0.4378
    0.0800    0.5001
    0.1000    0.5622
    0.1200    0.6240
    0.1400    0.6854
    0.1600    0.7462
    0.1800    0.8066
    0.2000    0.8663
    0.2200    0.9254
    0.2400    0.9839
    0.2600    1.0418
    0.2800    1.0991
    0.3000    1.1556
    0.3200    1.2111
    0.3400    1.2656
    0.3600    1.3190
    0.3800    1.3712
    0.4000    1.4223
    0.4200    1.4723
    0.4400    1.5212
    0.4600    1.5690
    0.4800    1.6156
    0.5000    1.6609
    0.5200    1.7048
    0.5400    1.7470
    0.5600    1.7878
    0.5800    1.8271
    0.6000    1.8650
    0.6200    1.9012
    0.6400    1.9356
    0.6600    1.9682
    0.6800    1.9993
    0.7000    2.0288
    0.7200    2.0567
    0.7400    2.0826
    0.7600    2.1063
    0.7800    2.1280
    0.8000    2.1480
    0.8200    2.1663
    0.8400    2.1828
    0.8600    2.1973
    0.8800    2.2096
    0.9000    2.2198
    0.9200    2.2278
    0.9400    2.2339
    0.9600    2.2379
    0.9800    2.2402
    1.0000    2.2409
/\relax}\relax
\setdashes <3pt>
\Green{\relax \putrule from 0.06 -1.8 to 0.15 -1.8 \relax}\relax
\put {$g_v$} [l] at 0.17 -1.8
\Green{\relax
\plot
         0    0.2500
    0.0200    0.4374
    0.0400    0.6229
    0.0600    0.8052
    0.0800    0.9829
    0.1000    1.1544
    0.1200    1.3180
    0.1400    1.4717
    0.1600    1.6144
    0.1800    1.7453
    0.2000    1.8636
    0.2200    1.9678
    0.2400    2.0563
    0.2600    2.1278
    0.2800    2.1820
    0.3000    2.2194
    0.3200    2.2402
    0.3400    2.2441
    0.3600    2.2306
    0.3800    2.1991
    0.4000    2.1495
    0.4200    2.0824
    0.4400    1.9986
    0.4600    1.8993
    0.4800    1.7858
    0.5000    1.6593
    0.5200    1.5209
    0.5400    1.3711
    0.5600    1.2109
    0.5800    1.0419
    0.6000    0.8661
    0.6200    0.6851
    0.6400    0.5005
    0.6600    0.3132
    0.6800    0.1249
    0.7000   -0.0624
    0.7200   -0.2464
    0.7400   -0.4255
    0.7600   -0.5986
    0.7800   -0.7648
    0.8000   -0.9226
    0.8200   -1.0700
    0.8400   -1.2051
    0.8600   -1.3269
    0.8800   -1.4347
    0.9000   -1.5279
    0.9200   -1.6055
    0.9400   -1.6660
    0.9600   -1.7089
    0.9800   -1.7340
    1.0000   -1.7423
/\relax}\relax
\endpicture
}
\setbox\figureeight=\vbox{\hsize=\xfiglen
\beginpicture
\eightrm
  \setcoordinatesystem units <\xfiglen,0.7\yfiglen>  point at 0 0.5
  \setplotarea x from 0 to 1, y from 0.5 to 2.2
  \axis bottom shiftedto y=0.5 ticks short numbered from 0  to 1 by 0.2 /
  \axis left ticks short numbered from 0.5 to 2 by 0.5 /
\put {{\small $a(x)$}} [lt] at 0.02 2.2
\setsolid
\plot  
    0.00    1.0000
    0.01    1.0631
    0.02    1.1259
    0.03    1.1876
    0.04    1.2470
    0.05    1.3034
    0.06    1.3558
    0.07    1.4035
    0.08    1.4457
    0.09    1.4819
    0.10    1.5115
    0.11    1.5342
    0.12    1.5497
    0.13    1.5578
    0.14    1.5586
    0.15    1.5520
    0.16    1.5384
    0.17    1.5181
    0.18    1.4915
    0.19    1.4592
    0.20    1.4219
    0.21    1.3802
    0.22    1.3351
    0.23    1.2873
    0.24    1.2378
    0.25    1.1875
    0.26    1.1374
    0.27    1.0885
    0.28    1.0417
    0.29    0.9980
    0.30    0.9581
    0.31    0.9230
    0.32    0.8933
    0.33    0.8697
    0.34    0.8528
    0.35    0.8430
    0.36    0.8406
    0.37    0.8460
    0.38    0.8591
    0.39    0.8800
    0.40    0.9085
    0.41    0.9443
    0.42    0.9871
    0.43    1.0363
    0.44    1.0914
    0.45    1.1516
    0.46    1.2162
    0.47    1.2842
    0.48    1.3549
    0.49    1.4271
    0.50    1.5000
    0.51    1.5725
    0.52    1.6435
    0.53    1.7122
    0.54    1.7774
    0.55    1.8384
    0.56    1.8942
    0.57    1.9441
    0.58    1.9873
    0.59    2.0233
    0.60    2.0515
    0.61    2.0716
    0.62    2.0833
    0.63    2.0864
    0.64    2.0810
    0.65    2.0670
    0.66    2.0448
    0.67    2.0147
    0.68    1.9771
    0.69    1.9326
    0.70    1.8819
    0.71    1.8256
    0.72    1.7647
    0.73    1.6999
    0.74    1.6322
    0.75    1.5625
    0.76    1.4918
    0.77    1.4211
    0.78    1.3513
    0.79    1.2834
    0.80    1.2181
    0.81    1.1564
    0.82    1.0989
    0.83    1.0463
    0.84    0.9992
    0.85    0.9580
    0.86    0.9230
    0.87    0.8946
    0.88    0.8727
    0.89    0.8574
    0.90    0.8485
    0.91    0.8457
    0.92    0.8487
    0.93    0.8569
    0.94    0.8698
    0.95    0.8866
    0.96    0.9066
    0.97    0.9288
    0.98    0.9525
    0.99    0.9765
    1.00    1.0000
 /
\setdots <2pt>
\footnotesize
\Red{\relax \putrule from 0.05 1.74 to 0.14 1.74 \relax}\relax
\put {$T=0.05$} [l] at 0.17 1.74
\Blue{\relax \putrule from 0.05 1.88 to 0.14 1.88 \relax}\relax
\put {$T=0.1$} [l] at 0.17 1.88
\Red{\relax
\plot
         0    0.8571
    0.0100    0.9201
    0.0200    0.9825
    0.0300    1.0434
    0.0400    1.1018
    0.0500    1.1569
    0.0600    1.2078
    0.0700    1.2538
    0.0800    1.2942
    0.0900    1.3285
    0.1000    1.3562
    0.1100    1.3770
    0.1200    1.3906
    0.1300    1.3968
    0.1400    1.3958
    0.1500    1.3875
    0.1600    1.3722
    0.1700    1.3503
    0.1800    1.3222
    0.1900    1.2885
    0.2000    1.2498
    0.2100    1.2069
    0.2200    1.1605
    0.2300    1.1114
    0.2400    1.0607
    0.2500    1.0091
    0.2600    0.9577
    0.2700    0.9073
    0.2800    0.8589
    0.2900    0.8135
    0.3000    0.7717
    0.3100    0.7345
    0.3200    0.7025
    0.3300    0.6764
    0.3400    0.6567
    0.3500    0.6439
    0.3600    0.6383
    0.3700    0.6401
    0.3800    0.6496
    0.3900    0.6667
    0.4000    0.6912
    0.4100    0.7229
    0.4200    0.7614
    0.4300    0.8062
    0.4400    0.8569
    0.4500    0.9127
    0.4600    0.9729
    0.4700    1.0367
    0.4800    1.1032
    0.4900    1.1714
    0.5000    1.2405
    0.5100    1.3094
    0.5200    1.3772
    0.5300    1.4428
    0.5400    1.5054
    0.5500    1.5641
    0.5600    1.6179
    0.5700    1.6662
    0.5800    1.7082
    0.5900    1.7434
    0.6000    1.7712
    0.6100    1.7913
    0.6200    1.8034
    0.6300    1.8074
    0.6400    1.8032
    0.6500    1.7909
    0.6600    1.7707
    0.6700    1.7430
    0.6800    1.7081
    0.6900    1.6667
    0.7000    1.6193
    0.7100    1.5666
    0.7200    1.5095
    0.7300    1.4487
    0.7400    1.3852
    0.7500    1.3199
    0.7600    1.2536
    0.7700    1.1874
    0.7800    1.1222
    0.7900    1.0588
    0.8000    0.9980
    0.8100    0.9406
    0.8200    0.8874
    0.8300    0.8390
    0.8400    0.7958
    0.8500    0.7583
    0.8600    0.7268
    0.8700    0.7016
    0.8800    0.6828
    0.8900    0.6702
    0.9000    0.6637
    0.9100    0.6631
    0.9200    0.6681
    0.9300    0.6780
    0.9400    0.6923
    0.9500    0.7104
    0.9600    0.7313
    0.9700    0.7544
    0.9800    0.7786
    0.9900    0.8027
    1.0000    0.8245
/\relax}\relax
\setdashes <3pt>
\Blue{\relax 
\plot
         0    0.9230
    0.0100    0.9862
    0.0200    1.0490
    0.0300    1.1104
    0.0400    1.1696
    0.0500    1.2255
    0.0600    1.2774
    0.0700    1.3245
    0.0800    1.3660
    0.0900    1.4015
    0.1000    1.4304
    0.1100    1.4523
    0.1200    1.4671
    0.1300    1.4745
    0.1400    1.4745
    0.1500    1.4673
    0.1600    1.4530
    0.1700    1.4321
    0.1800    1.4050
    0.1900    1.3722
    0.2000    1.3343
    0.2100    1.2922
    0.2200    1.2466
    0.2300    1.1984
    0.2400    1.1484
    0.2500    1.0977
    0.2600    1.0471
    0.2700    0.9976
    0.2800    0.9502
    0.2900    0.9058
    0.3000    0.8652
    0.3100    0.8292
    0.3200    0.7986
    0.3300    0.7739
    0.3400    0.7559
    0.3500    0.7449
    0.3600    0.7412
    0.3700    0.7450
    0.3800    0.7566
    0.3900    0.7759
    0.4000    0.8027
    0.4100    0.8368
    0.4200    0.8777
    0.4300    0.9251
    0.4400    0.9782
    0.4500    1.0365
    0.4600    1.0993
    0.4700    1.1655
    0.4800    1.2344
    0.4900    1.3050
    0.5000    1.3763
    0.5100    1.4474
    0.5200    1.5172
    0.5300    1.5848
    0.5400    1.6491
    0.5500    1.7093
    0.5600    1.7645
    0.5700    1.8139
    0.5800    1.8569
    0.5900    1.8928
    0.6000    1.9212
    0.6100    1.9416
    0.6200    1.9538
    0.6300    1.9576
    0.6400    1.9530
    0.6500    1.9401
    0.6600    1.9192
    0.6700    1.8905
    0.6800    1.8544
    0.6900    1.8116
    0.7000    1.7627
    0.7100    1.7084
    0.7200    1.6495
    0.7300    1.5869
    0.7400    1.5214
    0.7500    1.4540
    0.7600    1.3857
    0.7700    1.3173
    0.7800    1.2499
    0.7900    1.1844
    0.8000    1.1215
    0.8100    1.0621
    0.8200    1.0069
    0.8300    0.9565
    0.8400    0.9115
    0.8500    0.8722
    0.8600    0.8391
    0.8700    0.8124
    0.8800    0.7921
    0.8900    0.7782
    0.9000    0.7706
    0.9100    0.7690
    0.9200    0.7730
    0.9300    0.7822
    0.9400    0.7959
    0.9500    0.8134
    0.9600    0.8339
    0.9700    0.8566
    0.9800    0.8806
    0.9900    0.9044
    1.0000    0.9257
/\relax}\relax
\endpicture
}
\setbox\figurenine=\vbox{\hsize=\xfiglen
\beginpicture
\eightrm
  \setcoordinatesystem units <1.1\xfiglen,0.018\yfiglen>  point at 0 -20
  \setplotarea x from 0 to 1, y from -20 to 60
  \axis bottom shiftedto y=0 ticks short numbered from 1 to 1 by 0.2 /
  \axis left ticks short from -20 to 60 by 20 /
  \put {{\eightrm -20}} [r] at -0.02 -20
  \put {{\eightrm 0}} [r] at -0.02 0
  \put {{\eightrm 20}} [r] at -0.02 20
  \put {{\eightrm 40}} [r] at -0.02 40
  \put {{\eightrm 60}} [r] at -0.02 60
\setsolid
\small
\Blue{\relax \putrule from 0.14 55 to 0.22 55 \relax}\relax
\put {{\footnotesize $W$}} [l] at 0.25 55
\setdashes <3pt>
\Red{\relax \putrule from 0.14 50 to 0.22 50 \relax}\relax
\put {{\footnotesize $W'$}} [l] at 0.25 50
\setsolid
\Blue{\relax 
\plot
         0   -2.0786
    0.0100   -1.9483
    0.0200   -1.5509
    0.0300   -0.9328
    0.0400   -0.2072
    0.0500    0.4749
    0.0600    0.9678
    0.0700    1.1759
    0.0800    1.0833
    0.0900    0.7596
    0.1000    0.3370
    0.1100   -0.0323
    0.1200   -0.2253
    0.1300   -0.1868
    0.1400    0.0552
    0.1500    0.4028
    0.1600    0.7254
    0.1700    0.9069
    0.1800    0.8864
    0.1900    0.6781
    0.2000    0.3639
    0.2100    0.0618
    0.2200   -0.1186
    0.2300   -0.1155
    0.2400    0.0628
    0.2500    0.3434
    0.2600    0.6163
    0.2700    0.7768
    0.2800    0.7640
    0.2900    0.5833
    0.3000    0.3036
    0.3100    0.0303
    0.3200   -0.1340
    0.3300   -0.1286
    0.3400    0.0429
    0.3500    0.3142
    0.3600    0.5814
    0.3700    0.7423
    0.3800    0.7343
    0.3900    0.5582
    0.4000    0.2774
    0.4100   -0.0061
    0.4200   -0.1891
    0.4300   -0.2061
    0.4400   -0.0532
    0.4500    0.2114
    0.4600    0.4891
    0.4700    0.6778
    0.4800    0.7101
    0.4900    0.5776
    0.5000    0.3334
    0.5100    0.0716
    0.5200   -0.1085
    0.5300   -0.1393
    0.5400   -0.0105
    0.5500    0.2283
    0.5600    0.4864
    0.5700    0.6664
    0.5800    0.7007
    0.5900    0.5766
    0.6000    0.3403
    0.6100    0.0792
    0.6200   -0.1115
    0.6300   -0.1641
    0.6400   -0.0638
    0.6500    0.1467
    0.6600    0.3842
    0.6700    0.5559
    0.6800    0.5952
    0.6900    0.4857
    0.7000    0.2666
    0.7100    0.0180
    0.7200   -0.1705
    0.7300   -0.2325
    0.7400   -0.1496
    0.7500    0.0430
    0.7600    0.2704
    0.7700    0.4469
    0.7800    0.5084
    0.7900    0.4354
    0.8000    0.2592
    0.8100    0.0483
    0.8200   -0.1184
    0.8300   -0.1818
    0.8400   -0.1247
    0.8500    0.0227
    0.8600    0.1960
    0.8700    0.3219
    0.8800    0.3470
    0.8900    0.2577
    0.9000    0.0852
    0.9100   -0.1080
    0.9200   -0.2528
    0.9300   -0.2999
    0.9400   -0.2386
    0.9500   -0.1000
    0.9600    0.0568
    0.9700    0.1685
    0.9800    0.1930
    0.9900    0.1255
    1.0000    0.0000
 /\relax}\relax
\setdashes <4pt>
\Red{\relax 
\plot
         0         0
    0.0100   18.5036
    0.0200   36.9783
    0.0300   50.8391
    0.0400   56.3664
    0.0500   51.8950
    0.0600   38.4131
    0.0700   19.3470
    0.0800   -0.4206
    0.0900  -15.9973
    0.1000  -23.9265
    0.1100  -23.1642
    0.1200  -15.2725
    0.1300   -3.7801
    0.1400    7.0611
    0.1500   13.6899
    0.1600   14.3431
    0.1700    9.4927
    0.1800    1.4963
    0.1900   -6.3805
    0.2000  -11.1970
    0.2100  -11.3782
    0.2200   -7.2009
    0.2300   -0.5771
    0.2400    5.7543
    0.2500    9.3047
    0.2600    8.7736
    0.2700    4.5007
    0.2800   -1.7184
    0.2900   -7.3651
    0.3000  -10.1889
    0.3100   -9.0751
    0.3200   -4.4610
    0.3300    1.8484
    0.3400    7.3880
    0.3500    9.9860
    0.3600    8.5981
    0.3700    3.7054
    0.3800   -2.8701
    0.3900   -8.6574
    0.4000  -11.4711
    0.4100  -10.2331
    0.4200   -5.3739
    0.4300    1.3346
    0.4400    7.4481
    0.4500   10.7643
    0.4600   10.1392
    0.4700    5.9001
    0.4800   -0.2951
    0.4900   -6.0937
    0.5000   -9.3369
    0.5100   -8.8640
    0.5200   -4.9343
    0.5300    0.8906
    0.5400    6.3518
    0.5500    9.3527
    0.5600    8.7468
    0.5700    4.7583
    0.5800   -1.1179
    0.5900   -6.6919
    0.6000   -9.9096
    0.6100   -9.6222
    0.6200   -6.0074
    0.6300   -0.4866
    0.6400    4.8238
    0.6500    7.9160
    0.6600    7.6433
    0.6700    4.1382
    0.6800   -1.2535
    0.6900   -6.4927
    0.7000   -9.6177
    0.7100   -9.4737
    0.7200   -6.1302
    0.7300   -0.8347
    0.7400    4.4724
    0.7500    7.8905
    0.7600    8.2611
    0.7700    5.5748
    0.7800    0.9380
    0.7900   -3.8760
    0.8000   -7.1116
    0.8100   -7.6864
    0.8200   -5.5721
    0.8300   -1.7637
    0.8400    2.1375
    0.8500    4.5600
    0.8600    4.5653
    0.8700    2.1908
    0.8800   -1.5951
    0.8900   -5.2949
    0.9000   -7.4771
    0.9100   -7.3265
    0.9200   -4.9390
    0.9300   -1.2519
    0.9400    2.3516
    0.9500    4.5998
    0.9600    4.8275
    0.9700    3.2128
    0.9800    0.6677
    0.9900   -1.5654
    1.0000   -2.4444
 /\relax}\relax
\endpicture
}

%% file: fonts.tex
%
%
%
\font\tenrm=cmr10
\font\teni=cmmi10 \skewchar\teni='177
\font\tensy=cmsy10 \skewchar\tensy='60
\font\tenex=cmex10
\font\tenit=cmti10
\font\tensl=cmsl10
\font\tenbf=cmbx10
\font\tentt=cmtt10
\font\ninerm=cmr9
\font\ninei=cmmi9 \skewchar\ninei='177
\font\ninesy=cmsy9 \skewchar\ninesy='60
\font\nineit=cmti9
\font\ninesl=cmsl9
\font\ninebf=cmbx9
\font\ninett=cmtt9
\font\eightrm=cmr8
\font\eighti=cmmi8 \skewchar\eighti='177
\font\eightsy=cmsy8 \skewchar\eightsy='60
\font\eightit=cmti8
\font\eightsl=cmsl8
\font\eightbf=cmbx8
\font\eighttt=cmtt8
\font\sevenrm=cmr7
\font\seveni=cmmi7 \skewchar\seveni='177
\font\sevensy=cmsy7 \skewchar\sevensy='60
\font\sevenbf=cmbx7
\font\sevenit=cmmi7
\font\sevensl=cmmi7
\font\seventt=cmr7
\font\sixrm=cmr6
\font\sixi=cmmi6 \skewchar\sixi='177
\font\sixsy=cmsy6 \skewchar\sixsy='60
\font\sixbf=cmbx6
\font\fiverm=cmr5
\font\fivei=cmmi5 \skewchar\fivei='177
\font\fivesy=cmsy5 \skewchar\fivesy='60
\font\fivebf=cmbx5
\def\tenpoint{\def\rm{\fam0\tenrm}%
        \textfont0=\tenrm \scriptfont0=\sevenrm \scriptscriptfont0=\fiverm
        \textfont1=\teni \scriptfont1=\seveni \scriptscriptfont1=\fivei
        \textfont2=\tensy \scriptfont2=\sevensy \scriptscriptfont2=\fivesy
        \textfont3=\tenex \scriptfont3=\tenex \scriptscriptfont3=\tenex
        \def\it{\fam\itfam\tenit}%
        \textfont\itfam=\tenit
        \def\sl{\fam\slfam\tensl}%
        \textfont\slfam=\tensl
        \def\bf{\fam\bffam\tenbf}%
        \textfont\bffam=\tenbf \scriptfont\bffam=\sevenbf
                \scriptscriptfont\bffam=\fivebf
        \def\tt{\fam\ttfam\tentt}%
        \textfont\ttfam=\tentt
        \normalbaselineskip=12pt%
        \let\sc=\eightrm        
        \setbox\strutbox=\hbox{\vrule height8.5pt depth3.5pt width0pt}%
        \normalbaselines\rm}
\def\ninepoint{\def\rm{\fam0\ninerm}%
        \textfont0=\ninerm \scriptfont0=\sixrm \scriptscriptfont0=\fiverm
        \textfont1=\ninei \scriptfont1=\sixi \scriptscriptfont1=\fivei
        \textfont2=\ninesy \scriptfont2=\sixsy \scriptscriptfont2=\fivesy
        \textfont3=\tenex \scriptfont3=\tenex \scriptscriptfont3=\tenex
        \def\it{\fam\itfam\nineit}%
        \textfont\itfam=\nineit
        \def\sl{\fam\slfam\ninesl}%
        \textfont\slfam=\ninesl
        \def\bf{\fam\bffam\ninebf}%
        \textfont\bffam=\ninebf \scriptfont\bffam=\sixbf
                \scriptscriptfont\bffam=\fivebf
        \def\tt{\fam\ttfam\ninett}%
        \textfont\ttfam=\ninett
        \normalbaselineskip=11pt%
        \let\sc=\sevenrm        
        \setbox\strutbox=\hbox{\vrule height8pt depth3pt width0pt}%
        \normalbaselines\rm}
\def\eightpoint{\def\rm{\fam0\eightrm}%
        \textfont0=\eightrm \scriptfont0=\sixrm \scriptscriptfont0=\fiverm
        \textfont1=\eighti \scriptfont1=\sixi \scriptscriptfont1=\fivei
        \textfont2=\eightsy \scriptfont2=\sixsy \scriptscriptfont2=\fivesy
        \textfont3=\tenex \scriptfont3=\tenex \scriptscriptfont3=\tenex
        \def\it{\fam\itfam\eightit}%
        \textfont\itfam=\eightit
        \def\sl{\fam\slfam\eightsl}%
        \textfont\slfam=\eightsl
        \def\bf{\fam\bffam\eightbf}%
        \textfont\bffam=\eightbf \scriptfont\bffam=\sixbf
                \scriptscriptfont\bffam=\fivebf
        \def\tt{\fam\ttfam\eighttt}%
        \textfont\ttfam=\eighttt
        \normalbaselineskip=9pt%
        \let\sc=\sixrm  
        \setbox\strutbox=\hbox{\vrule height7pt depth2pt width0pt}%
        \normalbaselines\rm}
\def\sevenpoint{\def\rm{\fam0\sevenrm}%
        \textfont0=\sevenrm \scriptfont0=\fiverm \scriptscriptfont0=\fiverm
        \textfont1=\seveni \scriptfont1=\fivei \scriptscriptfont1=\fivei
        \textfont2=\sevensy \scriptfont2=\fivesy \scriptscriptfont2=\fivesy
        \textfont3=\tenex \scriptfont3=\tenex \scriptscriptfont3=\tenex
        \def\it{\fam\itfam\sevenit}%
        \textfont\itfam=\sevenit
        \def\sl{\fam\slfam\sevensl}%
        \textfont\slfam=\sevensl
        \def\bf{\fam\bffam\sevenbf}%
        \textfont\bffam=\sevenbf \scriptfont\bffam=\fivebf
                \scriptscriptfont\bffam=\fivebf
        \def\tt{\fam\ttfam\seventt}%
        \textfont\ttfam=\seventt
        \normalbaselineskip=8pt%
        \let\sc=\fiverm  
        \setbox\strutbox=\hbox{\vrule height6pt depth2pt width0pt}%
        \normalbaselines\rm}